\documentclass[12pt]{amsart}

\usepackage{amssymb,latexsym}

\usepackage{enumerate}

\makeatletter

\@namedef{subjclassname@2010}{

  \textup{2010} Mathematics Subject Classification}

\makeatother
\newtheorem{thm}{Theorem}[section]

\newtheorem{lem}[thm]{Lemma}
\newtheorem{pro}[thm]{Proposition}
\theoremstyle{definition}

\numberwithin{equation}{section}

\newcommand{\re}{\textup{Re}}
\newcommand{\im}{\textup{Im}}

\newcommand{\mb}{\mathbb}
\newcommand{\mc}{\mathcal}
\newcommand{\mbf}{\mathbf}
\newcommand{\e}{\varepsilon}

\renewcommand{\bar}{\overline}

\begin{document}

\baselineskip=17pt

\title[Large odd order character sums]{Large odd order character sums and improvements of the P\'{o}lya-Vinogradov inequality}

\author{Youness Lamzouri}

\author{Alexander P.  Mangerel}

\address{Department of Mathematics and Statistics,
York University,
4700 Keele Street,
Toronto, ON,
M3J1P3
Canada}

\email{lamzouri@mathstat.yorku.ca}

\address{Department of Mathematics\\ University of Toronto\\
Toronto, Ontario, Canada}
\email{sacha.mangerel@mail.utoronto.ca}

\date{}

\begin{abstract}  
For a primitive Dirichlet character $\chi$ modulo $q$, we define $M(\chi)=\max_{t } |\sum_{n \leq t} \chi(n)|$. In this paper, we study this quantity for characters of a fixed odd order $g\geq 3$. Our main result provides a further improvement of the classical P\'{o}lya-Vinogradov inequality in this case. More specifically, we show that for any such character $\chi$ we have
$$M(\chi)\ll_{\varepsilon} \sqrt{q}(\log q)^{1-\delta_g}(\log\log q)^{-1/4+\varepsilon},$$
where $\delta_g := 1-\frac{g}{\pi}\sin(\pi/g)$. This 
improves upon the works of Granville and Soundararajan and of Goldmakher. Furthermore, assuming the Generalized Riemann hypothesis (GRH) we prove that
$$
M(\chi) \ll \sqrt{q} \left(\log_2 q\right)^{1-\delta_g} \left(\log_3 q\right)^{-\frac{1}{4}}\left(\log_4 q\right)^{O(1)},
$$
where $\log_j$ is the $j$-th iterated logarithm. We also show unconditionally that this bound is best possible (up to a power of $\log_4 q$). One of the key ingredients in the proof of the upper bounds is a new Hal\'asz-type inequality for logarithmic mean values of completely multiplicative functions, which might be of independent interest.
\end{abstract}

\subjclass[2010]{Primary 11L40.}

\thanks{The first author is partially supported by a Discovery Grant from the Natural Sciences and Engineering Research Council of Canada.}

\maketitle
\section{Introduction}
The study of Dirichlet characters and their sums has been a central topic in analytic number theory for a long time. Let $q\geq 2$ and  $\chi$ be a non-principal Dirichlet character modulo $q$. An important quantity associated to $\chi$ is 
$$M(\chi) := \max_{t \leq q} \left|\sum_{n \leq t} \chi(n) \right|.$$
The best-known upper bound for $M(\chi)$, obtained independently by P\'olya and Vinogradov in 1918, reads
\begin{equation}\label{PVORIG}
M(\chi)  \ll \sqrt{q} \log q.
\end{equation}
Though one can establish this inequality using only basic Fourier analysis, improving on it has proved to be a difficult problem, and resisted substantial progress for several decades. Conditionally on the Generalized Riemann Hypothesis (GRH),  Montgomery and Vaughan \cite{MV2} showed in 1977 that
\begin{equation}\label{MVCHARBOUND}
M(\chi)\ll \sqrt{q}\log\log q.
\end{equation}
This bound is best possible in view of an old result of Paley \cite{Pa} that there exists an infinite family of primitive quadratic characters $\chi \bmod q$ such that 
\begin{equation}\label{Paley}
M(\chi) \gg \sqrt{q} \log \log q.
\end{equation} 
Assuming GRH, Granville and Soundararajan \cite{GrSo2} extended Paley's result to characters of a fixed even order $2k\geq 4$.  The assumption of GRH was later removed by Goldmakher and Lamzouri \cite{GL2}, who obtained this result unconditionally, and subsequently Lamzouri \cite{LAM} obtained the optimal implicit constant in \eqref{Paley} for even order characters. 

The situation is quite different for odd order characters. In this case, Granville and Soundararajan \cite{GrSo2} proved the remarkable result that both the P\'{o}lya-Vinogradov and the Montgomery-Vaughan bounds can be improved. More specifically, if $g\geq 3$ is an odd integer, and $\chi$ is a primitive character of order $g$ and conductor $q$ then they showed that
\begin{equation} \label{GSUP}
M(\chi) \ll \sqrt{q} (\log Q)^{1-\frac{\delta_g}{2} + o(1)},
\end{equation}
where $\delta_g := 1-\frac{g}{\pi}\sin(\pi/g)$ and 
\begin{equation}\label{THEQ}
Q := \begin{cases}  q &\text{ unconditionally}, \\  \log q &\text{ on GRH}. \end{cases}
\end{equation}
By refining their method, Goldmakher \cite{GOLD} was able to obtain the improved bound
\begin{equation} \label{GOLDSUP}
M(\chi) \ll \sqrt{q} (\log Q)^{1-\delta_g + o(1)}.
\end{equation}
Our first result gives a further improvement of the P\'olya-Vinogradov inequality for $M(\chi)$ when $\chi$ has odd order $g\geq 3$. Here and throughout, we write $\log_k x = \log(\log_{k-1} x)$ to denote the $k$th iterated logarithm, where $\log_1 x=\log x$.
\begin{thm}\label{MCHIUP1}
Let $g \geq 3$ be a fixed odd integer, and let $\e > 0$ be small. Then, for any primitive Dirichlet character $\chi$ of order $g$ and conductor $q$  we have
$$
M(\chi) \ll_{\e} \sqrt{q} \left(\log q\right)^{1-\delta_g}(\log\log  q)^{-\frac{1}{4}+ \e}.
$$
\end{thm}
\noindent The occurrence of $\e$ in the exponent of $\log\log q$ in the upper bound is a consequence of the possible existence of Siegel zeros. In particular, if Siegel zeros do not exist then the $(\log\log q)^{\e}$ term can be replaced by $(\log _3 q)^{O(1)}.$

Assuming GRH, and using results of Granville and Soundararajan (see Theorem \ref{GrSoGRH} below), Goldmakher \cite{GOLD} also showed that the conditional bound in \eqref{GOLDSUP} is best possible. More precisely, for every $\varepsilon>0$ and odd integer $g\geq 3$, he proved the existence of an infinite family of primitive characters $\chi\bmod q$ of order $g$ such that 
\begin{equation}\label{GSDOWN}
M(\chi) \gg_{\varepsilon} \sqrt{q}(\log\log q)^{1-\delta_g -\varepsilon},
\end{equation}
 conditionally on the GRH. By modifying the argument of Granville and Soundararajan and using ideas of Paley \cite{Pa}, Goldmakher and Lamzouri \cite{GL1} proved this result unconditionally. 

It is natural to ask to what degree of precision we can determine the exact order of magnitude of the maximal values of $M(\chi)$ when $\chi$ has odd order $g\geq 3$; in particular, can we determine the optimal $(\log\log q)^{o(1)}$ contributions in the conditional part of \eqref{GOLDSUP}, and in \eqref{GSDOWN}. We make progress in this direction by showing that this term can be replaced by $(\log_3 q)^{-\frac{1}{4}}(\log _4q)^{O(1)}$ in both  \eqref{GOLDSUP} and \eqref{GSDOWN}. This allows us to conditionally determine the maximal values of $M(\chi)$, up to a power of $\log_4 q.$ 
\begin{thm}\label{MCHIUP2}
Assume GRH. Let $g \geq 3$ be a fixed odd integer. Then for any primitive Dirichlet character $\chi$ of order $g$ and conductor $q$ we have
\begin{equation}\label{PREC0}
M(\chi) \ll \sqrt{q} \left(\log_2 q\right)^{1-\delta_g}(\log_3 q)^{-\frac{1}{4}}(\log_4q)^{O(1)}.
\end{equation}
\end{thm} 

\begin{thm} \label{MCHILOW}
Let $g\geq 3$ be a fixed odd integer. There are arbitrarily large $q$ and primitive Dirichlet characters $\chi$ modulo $q$ of order $g$ such that
\begin{equation} \label{PREC1}
M(\chi)  \gg \sqrt{q} \left(\log_2 q\right)^{1-\delta_g}  \left(\log_3 q\right)^{-\frac{1}{4}}\left(\log_4 q\right)^{O(1)}.
\end{equation}
\end{thm}

To obtain Theorem \ref{MCHILOW}, our argument relates $M(\chi)$ to the values of certain associated Dirichlet $L$-functions at $1$, and uses zero-density results and ideas from \cite{LAM} to construct characters $\chi$ for which these values are large. We shall discuss the different ingredients in the proofs of Theorems \ref{MCHIUP1}, \ref{MCHIUP2}  and \ref{MCHILOW} in detail in the next section. 
 
Recent progress on character sums was made possible by Granville and Soundararajan's discovery of a hidden structure among the characters $\chi$ having large $M(\chi)$. In particular, they show that $M(\chi)$ is large only when $\chi$ \emph{pretends} to be a character of small conductor and opposite parity. To define this notion of \emph{pretentiousness}, we need some notation.  Here and throughout we denote by $\mathcal{F}$ the class of completely multiplicative functions $f$ such that $|f(n)|\leq 1$ for all $n$. For $f, g\in \mc{F}$ we define
\begin{equation*}
\mb{D}(f,g;y) := \left(\sum_{p \leq y} \frac{1-\text{Re}(f(p)\bar{g(p)})}{p}\right)^{\frac{1}{2}},
\end{equation*}
which turns out to be a pseudo-metric on $\mc{F}$ (see \cite{GrSo2}). 
We say that $f$ \emph{pretends} to be $g$ (up to $y$) if there is a constant $0\leq \delta<1$ such that $\mb{D}(f,g;y)^2\leq \delta \log\log  y$. 

One of the key ingredients in the proof of \eqref{GSUP} is the following bound for logarithmic mean values of functions $f\in \mc{F}$ in terms of $\mb{D}(f,1;x)$ (see Lemma 4.3 of \cite{GrSo2}) 
\begin{equation}\label{BOUNDLOGDISTANCE}
\sum_{n\leq x} \frac{f(n)}{n}\ll (\log x) \exp\left(-\frac{1}{2}\mb{D}(f,1;x)^2\right).
\end{equation}
Note that the factor $1/2$ inside the exponential on the right hand side of \eqref{BOUNDLOGDISTANCE} is responsible for the weaker exponent $\delta_g/2$ in \eqref{GSUP}.

Goldmakher \cite{GOLD} realized that one can obtain the optimal exponent $\delta_g$ in \eqref{GOLDSUP} by replacing \eqref{BOUNDLOGDISTANCE} by a Hal\'{a}sz-type inequality for logarithmic mean values of multiplicative functions due to Montgomery and Vaughan \cite{MV}. Combining Theorem 2 of \cite{MV} with refinements of Tenenbaum (see Chapter III.4 of \cite{Te}) he deduced that (see Theorem 2.4 in \cite{GOLD})
\begin{equation} \label{HMTFIRST}
 \sum_{n \leq x} \frac{f(n)}{n} \ll (\log x)\exp\big(-\mc{M}(f;x,T)\big) + \frac{1}{\sqrt{T}},
\end{equation}
for all $f\in \mc{F}$ and $T\geq 1$, where 
$$
\mc{M}(f;x,T) := \min_{|t| \leq T} \mb{D}(f,n^{it};x)^2.
$$

Motivated by our investigation of character sums, we are interested in characterizing the functions $f\in \mc{F}$ that have a \emph{large} 
logarithmic mean, in the sense that
\begin{equation}\label{LARGELOG}
\sum_{n \leq x} \frac{f(n)}{n}\gg (\log x)^{\alpha},
\end{equation}
for some $0<\alpha\leq 1$.
Taking $T=1$ in \eqref{HMTFIRST} shows that this happens only when $f$ pretends to be $n^{it}$ for some $|t|\leq 1$. However, observe that
 $$\sum_{n \leq x} \frac{n^{it}}{n} = \frac{x^{it}-2^{it}}{it} +O(1)\asymp \min\left(\frac{1}{|t|}, \log x\right),$$
and hence $f(n)=n^{it}$ satisfies \eqref{LARGELOG} only when $|t|\ll (\log x)^{-\alpha}$. By refining the ideas of Montgomery and Vaughan \cite{MV} and Tenenbaum \cite{Te}, we prove the following result, which shows that this is essentially the only case.  
\begin{thm}\label{LogarithmicMean}
Let $f\in \mathcal{F}$ and $x\geq 2$. Then, for any real number $0< T \leq 1$ we have
$$\sum_{n\leq x}\frac{f(n)}{n}\ll  (\log x) \exp\big(-\mc{M}(f;x, T)\big)+\frac{1}{T},
$$
where the implicit constant is absolute.
\end{thm}
Taking $T=c(\log x)^{-\alpha}$ in this result (where $c>0$ is a suitably small constant), we deduce that if $f\in \mc{F}$ satisfies \eqref{LARGELOG}, then $f$ pretends to be $n^{it}$ for some $|t|\ll (\log x)^{-\alpha}$. Theorem \ref{LogarithmicMean} will be one of the key ingredients in obtaining our superior bounds for $M(\chi)$ in Theorems \ref{MCHIUP1} and \ref{MCHIUP2}.

\section{Detailed statement of results} 
To explain the key ideas in the proofs of Theorems \ref{MCHIUP1}, \ref{MCHIUP2} and \ref{MCHILOW}, we shall first sketch the argument of Granville and Soundararajan \cite{GrSo2}. Their starting point is P\'olya's Fourier expansion (see section 9.4 of \cite{MVbook}) for the character sum
$\sum_{n\leq t}\chi(n)$, which reads 
\begin{equation}\label{Polya}
\sum_{n\leq t}\chi(n)
=\frac{\tau(\chi)}{2\pi i}
	\sum_{1\leq |n|\leq N} \frac{\overline{\chi}(n)}{n}
		\left(1-e\left(-\frac{nt}{q}\right)\right)
		+O\left(1+\frac{q\log q}{N}\right),
\end{equation}
where $\chi$ is a primitive character modulo $q$,  $e(x) := e^{2\pi i x}$ and $\tau(\chi)$ is the Gauss sum 
$$
 \tau(\chi) := \sum_{n=1}^q \chi(n) e\Big(\frac{n}{q}\Big).
$$
Note that $|\tau(\chi)| = \sqrt{q}$ whenever $\chi$ is primitive.

Thus, in order to estimate $M(\chi)$, one needs to understand the size of the exponential sum
\begin{equation}\label{EXPCHARA}
\sum_{1\leq |n|\leq q} \frac{\chi(n)}{n}e(n\theta),
\end{equation}
for $\theta\in [0, 1]$. Montgomery and Vaughan \cite{MV2} showed that this sum is small if $\theta$ belongs to a \emph{minor arc}, i.e., $\theta$ can only be well-approximated by rationals with large denominators (compared to $q$). This leaves the more difficult case of $\theta$ lying in a \emph{major arc}. In this case, $\theta$ can be well-approximated by some rational $b/r$ with suitably small $r$ (compared to $q$). Granville and Soundararajan showed that in this case there is some large $N$ (depending on $\theta$, $b$, $r$ and $q$) such that we can approximate the sum 
\eqref{EXPCHARA} by 
\begin{align*}
&\sum_{1\leq |n| \leq N} \frac{\chi(n)}{n}e(bn/r) = \sum_{a\bmod r} e(ab/r) \sum_{1\leq |n| \leq N \atop n \equiv a \bmod r} \frac{\chi(n)}{n} \\
&=\frac{1}{\phi(r)} \sum_{\psi \bmod r} \left(\sum_{a\bmod r}\bar{\psi}(a)e(ab/r)\right) \sum_{1\leq |n| \leq N} \frac{\chi(n)\bar{\psi}(n)}{n}.
\end{align*}
The bracketed term, a Gauss sum, is well understood; in particular it has norm $\leq \sqrt{r^{\ast}}$, where $r^{\ast}$ is the conductor of $\psi$ (see e.g., Theorem 9.7 of \cite{MVbook}).   
Thus, what remains to be determined in order to bound $M(\chi)$, is an upper bound for the sum
\begin{equation}\label{LOGMEANPN}
\sum_{1\leq |n| \leq N} \frac{\chi(n)\bar{\psi}(n)}{n}
\end{equation}
for each character $\psi$ modulo $r$. Furthermore, observe that if $\chi$ and $\psi$ have the same parity then this sum is exactly $0$; hence, we only need to consider the case when $\chi$ and $\psi$ have opposite parities.

Granville and Soundararajan's breakthrough stems from their discovery of a ``repulsion'' phenomenon between characters $\chi$ of odd order (which are necessarily of even parity), and characters $\psi$ of odd parity and small conductor. A consequence of this phenomenon is that the sum \eqref{LOGMEANPN} is small, allowing them to improve the P\'olya-Vinogradov inequality in this case. 
More specifically, they show that if $\chi$ is a primitive character of odd order $g\geq 3$ and $\psi$ is an odd primitive character of conductor $m\leq (\log y)^A$ then  
 \begin{equation}\label{LOWERBOUDISTANCE}
 \mb{D}(\chi,\psi;y)^2\geq (\delta_g+o(1))\log\log y
 \end{equation}
(see Lemma 3.2 of \cite{GrSo2}). Inserting this bound in \eqref{BOUNDLOGDISTANCE} allows them to bound the sum \eqref{LOGMEANPN}, from which they deduce the unconditional case of \eqref{GSUP}. The proof of the conditional part of \eqref{GSUP} (when $Q=\log q$) proceeds along the same lines, but uses an additional ingredient, namely the following approximation for the sum \eqref{EXPCHARA} (see Proposition 2.3 and Lemma 5.2 of \cite{GrSo2}) conditional on GRH:
\begin{equation}\label{APPROXFRIABLE}
\sum_{n\leq q}\frac{\chi(n)}{n} e(n\theta)= \sum_{\substack{n\leq q \\ n \in \mc{S}(y)}}\frac{\chi(n)}{n} e(n\theta)+O\left(y^{-1/6}(\log q)^2\right).
\end{equation}
Here, $\mc{S}(y)$ is the set of $y$-\emph{friable} integers (also known as $y$-\emph{smooth} integers), i.e., the set of positive integers $n$ whose prime factors are all less than or equal to $y$.

In \cite{GOLD}, Goldmakher showed that the bound \eqref{LOWERBOUDISTANCE} is best possible. Furthermore, in order to obtain the exponent $\delta_g$ in \eqref{GOLDSUP}, he used the inequality \eqref{HMTFIRST} to bound the sum \eqref{LOGMEANPN} in terms of $\mc{M}(\chi\bar{\psi};y, T)$. However, to ensure that this argument works, one needs to show that the lower bound \eqref{LOWERBOUDISTANCE} still persists if we twist $\chi\bar{\psi}$ by Archimedean characters $n^{it}$ for $|t|\leq T$. By a careful analysis of 
$\mc{M}(\chi\bar{\psi};y,T)$, Goldmakher (see Theorem 2.10 of \cite{GOLD}) proved that (under the same assumptions as \eqref{LOWERBOUDISTANCE})
\begin{equation}\label{LBDGOLD}
\mc{M}(\chi\bar{\psi};y,(\log y)^2)\geq (\delta_g+o(1))\log\log y.
\end{equation}
Thus, by combining this bound with \eqref{HMTFIRST} and following closely the argument in \cite{GrSo2}, he was able to obtain \eqref{GOLDSUP}. 

In order to improve these results and establish Theorems \ref{MCHIUP1} and \ref{MCHIUP2}, the first step is to obtain more precise estimates for the quantity $\mc{M}(\chi\bar{\psi};y, T)$. We discover that there is a substantial difference between the sizes of $\mc{M}(\chi\bar{\psi}; y, T_1)$ and $\mc{M}(\chi\bar{\psi}; y, T_2)$ if $T_1$ is small and $T_2$ is large (a result that may be surprising in view of \eqref{LOWERBOUDISTANCE} and \eqref{LBDGOLD}). In fact, we prove that there is a large secondary term of size $(\log_2y)/k^2$ (where $k$ is the order of $\psi$) that appears in the estimate of $\mc{M}(\chi\bar{\psi}; y, T)$ when $T\leq (\log y)^{-c}$ (for some constant $c>0$), but disappears when $T\geq 1$. 
\begin{pro}\label{MINDIST}
Let $g\geq 3$ be a fixed odd integer, $\alpha\in (0, 1)$, and $\e>0$ be small. Let $\chi$ be a primitive character of order $g$ and conductor $q$. Let $\psi$ be an odd primitive character modulo $m$, with $m \leq (\log y)^{4\alpha/7}$. Put $k^{\ast} := k/(k,g)$. Then we have\begin{equation} \label{LowerBoundDistance2}
\mc{M}(\chi\bar{\psi}; y, (\log y)^{-\alpha}) \geq \left(\delta_g + \frac{\alpha \pi^2(1-\delta_g)}{4(gk^{\ast})^2}\right) \log_2 y - \beta \e \log m+O_{\alpha}\left(\log_2 m\right),
\end{equation}
where $\beta = 1$ if $m$ is an exceptional modulus and $\beta = 0$ otherwise.
\end{pro}
\begin{pro} \label{AD}
Assume GRH. Let $g\geq 3$ be a fixed odd integer. Let $N$ be large, and $y\leq (\log N)/10$. Let $\psi$ be an odd primitive character of conductor $m$ such that  $\exp\left(2\sqrt{\log_3 y}\right) \leq m\leq \exp\left(\sqrt{\log y}\right)$. Then, there exist at least $\sqrt{N}$  primitive characters $\chi$ of order $g$ and conductor $q\leq N$, such that for all $T\geq 1$ we have
\begin{equation*}
\mc{M}(\chi\bar{\psi}; y, T) \leq \delta_g \log_2 y + O\left(\log_2 m\right).
\end{equation*}
\end{pro}

The secondary term of size $\asymp (\log_2y)/k^2$ in the right hand side of \eqref{LowerBoundDistance2} is responsible for the additional saving of $(\log_2 Q)^{-1/4}$ (where $Q$ is defined in \eqref{THEQ}) in Theorems \ref{MCHIUP1} and \ref{MCHIUP2}; clearly, it does not appear in Proposition \ref{AD}, even in the range $m \ll (\log_2 y)^{\frac{1}{2}-\e}$ in which this secondary term is large.  Note that when $m$ is an exceptional modulus (see the precise definition in \eqref{EXCEPTIONALMODULUS} below), there is an additional term that appears when estimating $\mc{M}(\chi\bar{\psi}; y, (\log y)^{-\alpha})$  that has size $\log L(1, \chi_m)$, where $\chi_m$ is the exceptional character modulo $m$. In this case, the extra term $\varepsilon\log m$ on the right hand side of \eqref{LowerBoundDistance2} is due to Siegel's bound $L(1,\chi_m)\gg_{\e} m^{-\varepsilon}$.


To complete the proofs of Theorems \ref{MCHIUP1} and \ref{MCHIUP2}, we shall use our Theorem \ref{LogarithmicMean} to bound the sum \eqref{LOGMEANPN}, where we might choose $T=(\log y)^{-\alpha}$ to take advantage of Proposition \ref{MINDIST}. Note that in view of Proposition \ref{AD}, one loses the additional saving of $(\log_2 Q)^{-1/4}$ in Theorems \ref{MCHIUP1} and \ref{MCHIUP2} if one simply uses \eqref{HMTFIRST} with $T=(\log y)^2$, as in \cite{GOLD}. By using Theorem \ref{LogarithmicMean} and following the ideas in \cite{GrSo2}, we prove the following result, which is a refinement of Theorem 2.9 in \cite{GOLD}. This together with Proposition \ref{MINDIST} implies both Theorems \ref{MCHIUP1} and  \ref{MCHIUP2}.

\begin{thm}\label{COND}
Let $\chi$ be a primitive character modulo $q$, and let $Q$ be as in \eqref{THEQ}.  Of all primitive characters with conductor below $(\log Q)^{4/11}$, let $\xi$ modulo $m$ be that character for which $\mc{M}\left(\chi\bar{\xi};Q, (\log Q)^{-7/11}\right)$ is a minimum. Then we have
\begin{equation*}
M(\chi) \ll \Big(1-\chi(-1)\xi(-1)\Big)\frac{\sqrt{qm}}{\phi(m)} (\log Q) \exp\left(-\mc{M}\left(\chi\bar{\xi}; Q,(\log Q)^{-\frac{7}{11}}\right)\right) + \sqrt{q}\left(\log Q\right)^{\frac{9}{11} + o(1)}.
\end{equation*}
\end{thm}
Note that $\delta_g$ is decreasing as a function of $g$, so $1-\delta_g \geq 1-\delta_3\approx 0.827>9/11$ for all $g \geq 3$. Therefore, when $\chi$ is a primitive character of odd order $g\geq 3$ and conductor $q$, we get the better bound 
$ M(\chi)\ll \sqrt{q}\left(\log Q\right)^{\frac{9}{11} + o(1)}$, unless $\xi$ is odd and $\mc{M}\left(\chi\bar{\xi}; Q,(\log Q)^{-\frac{7}{11}}\right)$ is small.


We next discuss the ideas that go into the proof of Theorem \ref{MCHILOW}. 
To obtain \eqref{GSDOWN} under GRH, Goldmakher \cite{GOLD} used the following result from \cite{GrSo2}, which relates $M(\chi)$ to the distance between $\chi$ and any primitive character $\psi$ with small conductor and parity opposite to that of $\chi$. 
\begin{thm}[Theorem 2.5 of \cite{GrSo2}]\label{GrSoGRH}
Assume GRH. Let $\chi\bmod q$ and $\psi\bmod m$ be primitive characters such that $\chi(-1)=-\psi(-1)$. Then we have
$$
M(\chi) + \frac{\sqrt{qm}}{\phi(m)}\log_3q \gg \frac{\sqrt{qm}}{\phi(m)} (\log_2 q) \exp\left(-\mb{D}(\chi,\psi;\log q)^2\right).
$$
\end{thm}
\noindent Thus, it only remains to produce characters $\chi$ and $\psi$ which satisfy the assumptions of Theorem \ref{GrSoGRH}, and for which the lower bound \eqref{LOWERBOUDISTANCE} is attained when $y=\log q$.  Using the Eisenstein reciprocity law, Goldmakher (see Proposition 9.3 of \cite{GOLD}) proved that for any $\varepsilon>0$, there exists an odd primitive character $\psi$ modulo $m\ll_{\varepsilon}1$, and an infinite family of primitive characters $\chi\bmod q$ of order $g$ such that 
\begin{equation}\label{GOLDreciprocity}
\mb{D}(\chi,\psi;\log q)^2\leq (\delta_g+\varepsilon)\log_3q.
\end{equation}
To remove the assumption of GRH, Goldmakher and Lamzouri \cite{GL1} (see Theorem 1 of \cite{GL1}) used ideas of Paley \cite{Pa} to obtain a weaker version of Theorem \ref{GrSoGRH} unconditionally. Namely, they showed that if $\chi$ is odd and $\psi$ is even then
$$M(\chi) + \sqrt{q}\gg \frac{\sqrt{qm}}{\phi(m)}\left(\frac{\log_2 q}{\log_3 q}\right) \exp\left(-\mb{D}(\chi,\psi;\log q)^2\right).$$
Although this bound is enough to obtain \eqref{GSDOWN} unconditionally in view of \eqref{GOLDreciprocity}, it is not sufficient to yield the precise estimate in Theorem \ref{MCHILOW}, due to the loss of a factor of $\log_3 q$ over Theorem \ref{GrSoGRH}. \\
Using a completely different method, based on zero density estimates for Dirichlet $L$-functions, we recover the original bound of Granville and Soundararajan unconditionally for all characters $\chi$ modulo $q$ with $q\leq N$, except for a small \emph{exceptional} set of cardinality $\ll N^{\varepsilon}$. Our argument also gives a simple proof of Theorem \ref{GrSoGRH}, which exploits the natural properties of the values of Dirichlet $L$-functions at $1$, and avoids the difficult study of exponential sums with multiplicative functions (see Section 6 of \cite{GrSo2}). Note that the statement of Theorem \ref{GrSoGRH} trivially holds when $m > \log q$, since $\mb{D}(\chi,\psi;\log q)^2\ll \log_3 q$. We thus only need to consider the case $m\leq \log q$.
\begin{thm} \label{GL}
Let $\e > 0$ and let $N$ be large. Let $m \leq \log N$ be a positive integer and let $\psi$ be a primitive character modulo $m$.  Then, for all but at most $N^{\varepsilon}$ primitive characters $\chi$ modulo $q$ with   $q\leq N$ and such that $\chi(-1)=-\psi(-1)$ we have 
\begin{equation}\label{LowerBoundMchi}
M(\chi) + \sqrt{q} \gg_{\varepsilon} \frac{\sqrt{qm}}{\phi(m)} (\log_2 q) \exp\left(-\mb{D}(\chi,\psi;\log q)^2\right).
\end{equation}
Moreover, if we assume GRH, then \eqref{LowerBoundMchi} is valid for all primitive characters $\chi$ modulo $q$ with $q\leq N$, and the implicit constant in \eqref{LowerBoundMchi} is absolute.
\end{thm}
To complete the proof of Theorem \ref{MCHILOW}, we thus need to refine the estimate \eqref{GOLDreciprocity}, and this can be achieved using the same ideas as in the proof of Proposition \ref{MINDIST}. However, Goldmakher's proof of \eqref{GOLDreciprocity} only produces an infinite sequence of primitive characters $\chi$, and this is not enough to use in Theorem \ref{GL}, due to the possible existence of an exceptional set of characters for which \eqref{LowerBoundMchi} does not hold. To overcome this difficulty, we use the results of \cite{LAM} to prove the existence of \emph{many} primitive characters $\chi$ of order $g$ and conductor $q\leq N$ such that when $y \ll \log N$, $\mb{D}(\chi,\psi;y)$ is maximal.

\begin{pro} \label{UPPER}
Let $g\geq 3$ be a fixed odd integer. Let $N$ be large and $y\leq (\log N)/10$ be a real number. Let $m$ be a non-exceptional modulus such that $m\leq (\log y)^{4/7}$, and let $\psi$ be an odd primitive character of conductor $m$. Let $k$ be the order of $\psi$ and  put $k^{\ast}=k/(g,k)$.  Then, there exist at least $\sqrt{N}$ primitive characters $\chi$ of order $g$ and conductor $q\leq N$ such that
\begin{equation} \label{OPT}
\mb{D}(\chi,\psi;y)^2 =  \left(1-(1-\delta_g)\frac{\pi/gk^{\ast}}{\tan(\pi/gk^{\ast})}\right) \log_2y + O\left(\log_2 m\right).
\end{equation}
\end{pro}

\section{A lower bound for $M(\chi)$: Proof of Theorem \ref{GL}}

The main ingredient in the proof of  Theorem \ref{GrSoGRH} of \cite{GrSo2} is the approximation \eqref{APPROXFRIABLE}, which is valid under the assumption of GRH. 
To avoid this assumption, we shall instead relate $M(\chi)$ to the values of certain Dirichlet $L$-functions at $s=1$, and then use the classical zero-density estimates for these $L$-functions. \begin{pro}\label{CharSumL1}
Let $q$ be large and $m \leq q/(\log q)^2$. Let $\chi\bmod q$ and $\psi\bmod m$ be primitive characters such that $\psi(-1)=-\chi(-1)$. Then we have
$$ M(\chi)+\sqrt{q} \gg \frac{\sqrt{qm}}{\phi(m)} \cdot \left|L\left(1, \chi\overline{\psi}\right)\right|.$$
\end{pro}
We first need the following lemma.
\begin{lem} \label{PVAPP}
Let $q$ be large and $m \leq q/(\log q)^2$. Let $\chi$ be a character modulo $q$ and $\psi$ be a character modulo $m$ such that $\chi \bar{\psi}$ is non-principal. Then
\begin{equation*}
L(1,\chi\bar{\psi}) = \sum_{n \leq q} \frac{\chi(n)\bar{\psi}(n)}{n} + O(1).
\end{equation*}
\end{lem}
\begin{proof}
Note that $\chi\bar{\psi}$ is a non-principal character of conductor at most $qm\leq (q/\log q)^2$. Therefore, using partial summation and the P\'{o}lya-Vinogradov inequality we obtain
\begin{align*}
\sum_{q < n \leq N} \frac{\chi(n)\bar{\psi}(n)}{n} &= \sum_{q<n\leq N} \frac{1}{n(n+1)}\left(\sum_{q < k \leq n} \chi\bar{\psi}(k)\right) +O(1) \ll 1,
\end{align*}
and the claim follows.
\end{proof}
\begin{proof}[Proof of Proposition \ref{CharSumL1}]
Taking $N=q$ in \eqref{Polya} gives
$$M(\chi)+\log q\gg \sqrt{q} \cdot \max_{\theta}\left|\sum_{1\leq |n|\leq q} \frac{\chi(n)}{n}
		\left(1-e\left(n \theta\right)\right)\right|.$$
Moreover,  we observe that
\begin{align*}
\sum_{b\bmod m} \psi(b) \sum_{1\leq |n|\leq q} \frac{\chi(n)}{n}\left(1-e\left(\frac{nb}{m}\right)\right)
&=-\sum_{1\leq |n|\leq q}\frac{\chi(n)}{n}\sum_{b\bmod m} \psi(b)e\left(\frac{nb}{m}\right)\\
&= -\tau(\psi)\sum_{1\leq |n|\leq q}\frac{\chi(n)\bar{\psi}(n)}{n},\\
\end{align*}
which follows from the identity
$$ \sum_{b\bmod m} \psi(b)e\left(\frac{nb}{m}\right)= \bar{\psi}(n)\tau(\psi).$$
Since $\chi$ and $\psi$ are primitive and $m\leq q/(\log q)^2$ then $\chi\bar{\psi}$ is non-principal. Therefore, by Lemma \ref{PVAPP} together with the fact that $\chi\bar{\psi}(-1)=-1$ we deduce that
$$ \sum_{1\leq |n|\leq q}\frac{\chi(n)\bar{\psi}(n)}{n}= 2 \sum_{1\leq n\leq q}\frac{\chi(n)\bar{\psi}(n)}{n}
= 2 L(1,\chi\bar{\psi})+O(1).$$
The result follows upon noting that
$$ \left|\sum_{b\bmod m} \psi(b) \sum_{1\leq |n|\leq q} \frac{\chi(n)}{n}\left(1-e\left(\frac{nb}{m}\right)\right)\right|\leq \phi(m) \cdot \max_{\theta}\left|\sum_{1\leq |n|\leq q} \frac{\chi(n)}{n}
		\left(1-e\left(n \theta\right)\right)\right|,
$$
and that $|\tau(\psi)| = \sqrt{m}$ by the primitivity of $\psi$.
\end{proof}

In order to complete the proof of Theorem \ref{GL}, we need to approximate $L(1,\chi\bar{\psi})$ by a short truncation of its Euler product. Using zero density estimates, we prove that this is possible for almost all primitive characters $\chi$. 
\begin{pro}\label{BoundExcep}
Fix $0<\varepsilon<1$ and let $A=100/\varepsilon$.   Let $N$ be large and $m \leq \log N$. Then for all but at most $N^{\varepsilon}$ primitive characters $\chi$ modulo $q \leq N$ we have
\begin{equation}\label{ShortApproxL1chi}
L(1,\chi\bar{\psi}) = \left(1+O\left(\frac{1}{\log N}\right)\right)\prod_{p \leq \log^A N} \left(1-\frac{\chi(p)\bar{\psi(p)}}{p}\right)^{-1} .
\end{equation}
for all primitive characters $\psi$ modulo $m$. Moreover, if we assume GRH, then \eqref{ShortApproxL1chi} is valid with $A=10$, for all primitive characters $\chi$ modulo $q\leq N$ and $\psi$ modulo $m$.
\end{pro}
In order to prove this proposition, we first need some preliminary results.
\begin{lem}\label{GoodLFunctions} Let $q$ be  large and $\chi$ be a non-principal character modulo $q$. Let $2\leq T\leq q^2$ and $X\geq 2$. 
Let $\frac{1}{2} \leq \sigma_0 < 1$ and suppose that the
rectangle $\{ s:  \sigma_0 <\textup{Re}(s) \leq 1, \ \
|\textup{Im}(s)| \leq T+3\}$ does not contain any zeros of $L(s,\chi)$.
Then we have
$$
\log L(1, \chi)= -\sum_{p\leq X} \log \left(1-\frac{\chi(p)}{p}\right)+O\left(\frac{\log X}{T}+\frac{\log q}{(1-\sigma_0)T}+ \frac{\log q \log T}{(1-\sigma_0)^2} X^{(\sigma_0-1)/2}\right).
$$
\end{lem}
\begin{proof}
Let $\alpha=1/\log X$. Then it follows from Perron's formula that
\begin{equation}\label{Perron}
\begin{aligned}
&\frac{1}{2\pi i} \int_{\alpha-iT}^{\alpha+iT} \log L(1+s, \chi) \frac{X^s}{s} ds\\
& = \sum_{n\leq X} \frac{\Lambda(n)}{n\log n}\chi(n)+ O\left(\sum_{n=1}^{\infty} \frac{\Lambda(n)}{n^{1+\alpha}\log n}\min\left(1, \frac{1}{T\log|X/n|}\right)\right)\\
&= \sum_{n\leq X} \frac{\Lambda(n)}{n\log n}\chi(n)+ O\left(\frac{\log X}{T}+ \frac{1}{X}\right),
\end{aligned}
\end{equation}
by a standard estimation of the error term. Moreover, we observe that 
\begin{align*}
 \sum_{n\leq X} \frac{\Lambda(n)}{n\log n}\chi(n)&= -\sum_{p\leq X} \log \left(1-\frac{\chi(p)}{p}\right)+O\left(\sum_{k=2}^{\infty}\sum_{p^k>X} \frac{1}{k p^k}\right)\\
 &=-\sum_{p \leq X} \log\left(1-\frac{\chi(p)}{p}\right) + O\left(X^{-\frac{1}{2}}\right).
\end{align*}
We now move the contour in \eqref{Perron} to the line $\re(s)=\sigma_1-1$, where $\sigma_1=(1+\sigma_0)/2$. We encounter a simple pole at $s=0$ that leaves a residue of $\log L(1, \chi)$. Furthermore, it follows from Lemma 8.1 of \cite{GrSo0} that for $\sigma\geq \sigma_1$ and $|t|\leq T$ we have 
$$\log L(\sigma+it, \chi)\ll \frac{\log q}{\sigma-\sigma_0}\ll\frac{\log q}{1-\sigma_0} .$$Therefore, we deduce that
$$ \frac{1}{2\pi i} \int_{\alpha-iT}^{\alpha+iT} \log L(1+s, \chi) \frac{X^s}{s} ds= \log L(1, \chi)+\mathcal{E}, $$
where 
\begin{align*}
\mathcal{E}&= \frac{1}{2\pi i} \left(\int_{\alpha-iT}^{\sigma_1-1-iT}+\int_{\sigma_1-1-iT}^{\sigma_1-1+iT}+ \int_{\sigma_1-1+iT}^{\alpha+iT}\right)\log L(1+s, \chi) \frac{X^s}{s} ds\\
&\ll \frac{\log q}{(1-\sigma_0)T}+ \frac{\log q \log T}{(1-\sigma_0)^2} X^{(\sigma_0-1)/2}.
\end{align*}
Since $\sigma_0 \geq 1/2$, combining the above estimates completes the proof.
\end{proof}
\begin{lem}\label{InducedFixed}
Let $\xi\bmod q$ and $\psi \bmod m$ be primitive characters. Then, there is a unique primitive character $\chi$ such that $\chi\psi$ is induced by $\xi$ if $m \mid q$, and no such character exists if $m\nmid q$.
\end{lem}

\begin{proof}
Suppose that $\chi \psi$ is induced by $\xi$, where $\chi$ is a primitive character of conductor $\ell$. Then we must have $q=[\ell, m]$, and hence 
there is no such character $\chi$ if $m\nmid q$. 

Now, suppose that $m\mid q$, and let $m=p_1^{a_1}\cdots p_k^{a_k}$ be its prime factorization. We construct $\chi$ in this case as follows. Since $q=[\ell, m]$, then we have 
$q=q_0\cdot p_1^{b_1}\cdots p_k^{b_k}$ where $(q_0, m)=1$ and $b_j\geq a_j$ for all $1\leq j\leq k$, and $\ell=q_0\cdot p_1^{c_1}\cdots p_k^{c_k}$ where $c_j=b_j$ if $b_j>a_j$ and $0\leq c_j\leq a_j$ if $b_j=a_j$. 
Now, since $\xi$ is primitive then $\xi= \tilde{\xi}\cdot\xi_1\cdots\xi_k$ where $\tilde{\xi}$ is a primitive character modulo $q_0$ and $\xi_j$ is a primitive character modulo $p_j^{b_j}$ for $1\leq j\leq k$. Similarly, we have $\psi=\psi_1\cdots\psi_k$ and $\chi=\tilde{\chi}\cdot \chi_1\cdots\chi_k$ where $\tilde{\chi}$ is a primitive character modulo $q_0$ and $\psi_j, \chi_j$ are primitive characters modulo $p_j^{a_j}$ and $p_j^{c_j}$ respectively. Moreover, since $\xi$ induces $\chi\psi$ then we must have $\tilde{\chi}=\tilde{\xi}$, and $\xi_j$ induces $\chi_j\psi_j$ for all $1\leq j\leq k$. But this implies that $\chi_j(n)=\xi_j(n)\bar{\psi_j(n)}$ for all $n$ such that $p_j\nmid n$, and hence we deduce that there is only one choice for $\chi_j$ since it is  primitive. Since this holds for all $1\leq j\leq k$, the character $\chi$ is unique.
\end{proof}

\begin{proof}[Proof of Proposition \ref{BoundExcep}]
By Bombieri's classical zero-density estimate (see Theorem 20 of \cite{Bo}), we know that there are at most $N^{6(1-\sigma)}(\log N)^B$ primitive characters $\xi$ with conductor $q\leq N\log N$ and such that $L(s, \xi)$ has a zero in the rectangle $\{s: \sigma \leq \re(s) \leq 1, |\im(s)|\leq N\}$, where $B$ is an absolute constant. Let $\xi_1, \cdots, \xi_L$ be these characters with $\sigma=1-\varepsilon/20$. Then, it follows from the above argument that $L\ll N^{\varepsilon/2}$. 
 
Recall that if $\xi$ is a primitive character that induces $\tilde{\xi}$, then $L(s, \xi)$ and $L(s, \tilde{\xi})$ have the same zeros in the half-plane $\re(s)>0$.   For a primitive character $\psi$ modulo $m$, let $\mc{E}_{\psi}$ denote the set of primitive characters $\chi$ modulo $q$ with $q\leq N$ and such that $\chi \bar{\psi}$ is induced by one of the characters $\xi_j$ for $1\leq j\leq L$. Let $\mathcal{E}_m$ be the union over all primitive characters $\psi$ modulo $m$ of the sets $\mathcal{E}_{\psi}$. Then, it follows from Lemma \ref{InducedFixed} that
$$
\left|\mc{E}_m\right| \leq \sum_{\psi \bmod m \atop \psi \textup { primitive}} |\mc{E}_{\psi}| \leq L\phi(m) \ll N^{\varepsilon}.
$$
Let $X=(\log N)^{A}$ where $A=100/\varepsilon$.  If $\chi$ is a primitive character with conductor $q\leq N$ and such that $\chi\notin\mc{E}_m$ then it follows from Lemma \ref{GoodLFunctions} with $T=X$ that for all primitive characters $\psi$ modulo $m$ we have
$$
\log L(1,\chi\bar{\psi})= -\sum_{p\leq X} \log\left(1-\frac{\chi(p)\bar{\psi}(p)}{p}\right) + O\left(\frac{1}{\log N}\right),
$$
which implies \eqref{ShortApproxL1chi}.
Finally, if we assume GRH, then this estimate is valid for all primitive characters $\chi$ modulo $q\leq N$ and $\psi$ modulo $m$ with $X=(\log N)^{10}$ by Lemma \ref{GoodLFunctions}.
\end{proof}
We can now prove Theorem \ref{GL}.
\begin{proof}[Proof of Theorem \ref{GL}]
Combining Propositions \ref{CharSumL1} and \ref{BoundExcep} we deduce that for all but at most $N^{\varepsilon}$ primitive characters $\chi$ modulo $q$ with $N^{\varepsilon/3} \leq q \leq N$ we have
\begin{equation}\label{LowerBoundMCHI}
M(\chi) + \sqrt{q}  \gg \frac{\sqrt{qm}}{\phi(m)} \prod_{p \leq \log^A N} \left(1-\frac{\chi(p)\bar{\psi(p)}}{p}\right)^{-1} 
\end{equation}
with $A=100/\varepsilon$. The first part of the theorem follows, upon noting that
$$ \prod_{p \leq \log^A N} \left(1-\frac{\chi(p)\bar{\psi(p)}}{p}\right)^{-1}  \gg_{\varepsilon} (\log_2 q) \cdot \exp\big(-\mb{D}(\chi,\psi;\log q)^2\big).
$$
The second part follows along the same lines, since if we assume GRH then \eqref{LowerBoundMCHI} holds with $A=10$ for all primitive characters $\chi$ with conductor $q\leq N$.
\end{proof}

\section{Estimates for the distance $\mb{D}(\chi,\psi;y)$: Proofs of Proposition \ref{UPPER} and Theorem \ref{MCHILOW}}
We shall first prove a lower bound for $\mb{D}(\chi,\psi;y)^2$, which is a refined version of \eqref{LOWERBOUDISTANCE}, that shows that Proposition  \ref{UPPER} is best possible. This will also be the main ingredient in the proof of Proposition \ref{MINDIST}. 

\begin{pro} \label{UPPER2}
Let $g\geq 3$ be a fixed odd integer, and $\varepsilon>0$ be small. Let $\psi$ be an odd primitive character of conductor $m$ and order $k$, and $y$ be such that $m\leq (\log y)^{4/7}$. Put $k^{\ast}=k/(g,k)$.  Then, for any primitive character $\chi \pmod q$ of order $g$ we have
$$
\mb{D}(\chi,\psi;y)^2\geq \left(1-(1-\delta_g)\frac{\pi/gk^{\ast}}{\tan(\pi/gk^{\ast})}\right) \log_2y-\beta \varepsilon \log m+ O\left(\log_2 m\right),
$$
where $\beta=0$ if $m$ is a non-exceptional modulus, and $\beta=1$ if $m$ is exceptional.
\end{pro}
We say that $m\geq 1$ is an {\it exceptional} modulus if there exists a Dirichlet character $\chi_m$ and a complex number $s$ such that  $L(s,\chi_m)=0$ and 
\begin{equation}\label{EXCEPTIONALMODULUS}
\text{Re}(s)\geq 1-\frac{c}{\log (m(\text{Im}(s)+2))}
\end{equation} 
for some sufficiently small constant $c>0$. 
One expects that there are no such moduli, but what is known unconditionally is that if $m$ is exceptional, then there is only one {\it exceptional} character $\chi_m$ modulo $m$, which is quadratic, and for which $L(s,\chi_m)$ has a unique zero in the region \eqref{EXCEPTIONALMODULUS} which is real and simple (this zero is called a Siegel zero).

For $g \geq 3$, we let $\mu_g$ denote the set of $g$-th roots of unity. Then, we observe that
\begin{align}
\mb{D}(\chi,\psi;y)^2 &= \log\log y- \sum_{p \leq y} \frac{\text{Re}(\chi(p)\bar{\psi}(p))}{p} +O(1)  \nonumber\\
&\geq  \log \log  y - \sum_{\ell\bmod k} \max_{z \in \mu_g \cup \{0\}} \text{Re}\left(z \cdot e\left(-\frac{\ell}{k}\right)\right)\sum_{p \leq y \atop \psi(p)= e\left(\frac{\ell}{k}\right)} \frac{1}{p} +O(1).\label{DIST}
\end{align}
Proposition \ref{UPPER2} follows from this inequality together with Proposition \ref{PreciseSumMax} below, which provides an asymptotic formula for the sum on the right hand side of \eqref{DIST}. To establish Proposition \ref{UPPER}, we need an additional ingredient, namely that there exist many primitive characters $\chi$ whose values we can control at the small primes $p\leq c \log q$ so that the inequality in \eqref{DIST} is sharp for $y\leq c\log q$. This is proven in Lemma \ref{VEC} below. 
\begin{pro}\label{PreciseSumMax}
Let $g\geq 3$ be a fixed odd integer, and $\varepsilon>0$ be small. Let $\psi$ be an odd primitive character of conductor $m$ and order $k$, and $y$ be such that $m\leq (\log y)^{4/7}$. Put $k^{\ast}=k/(g,k)$. Then 
\begin{equation}\label{AsymptoticMaxSumRE}
\begin{aligned}
&\sum_{\ell\bmod k} \max_{z \in \mu_g \cup \{0\}} \textup{Re}\left(z \cdot e\left(-\frac{\ell}{k}\right)\right)\sum_{p \leq y \atop \psi(p)= e\left(\frac{\ell}{k}\right)} \frac{1}{p} \\
&= (1-\delta_g)\frac{\pi/gk^{\ast}}{\tan(\pi/gk^{\ast})}\log_2 y + \theta \varepsilon \log m+ O\left(\log_2 m\right),
\end{aligned}
\end{equation}
where $\theta=0$ if $m$ in a non-exceptional modulus, and $|\theta|\leq 1$ if $m$ is exceptional. 
\end{pro}

We first record the following lemma, which is a special case of Lemma 8.3 of \cite{GOLD}. 
\begin{lem} \label{MAX}
Let $g,k$ and $k^{\ast}$ be as in Proposition \ref{PreciseSumMax}. Then
\begin{equation*}
\frac{1}{k}\sum_{\ell\bmod k}\max_{z \in \mu_g \cup \{0\}} \textup{Re}\left(z \cdot e\left(-\frac{\ell}{k}\right)\right) = (1-\delta_g)\frac{\pi/gk^{\ast}}{\tan(\pi/gk^{\ast})}. 
\end{equation*}
\end{lem}
\begin{proof}
This is Lemma 8.4 of \cite{GOLD} (see also Lemma \ref{MAXGold} below) with $\theta=0$.
\end{proof}

In view of this lemma, our next task is to estimate the inner sum in the left hand side of \eqref{AsymptoticMaxSumRE}. Since $\psi$ is periodic modulo $m$ we have
\begin{equation} \label{ALLAS}
\sum_{p \leq y \atop \psi(p) = e\left(\frac{\ell}{k}\right)}\frac{1}{p} = \sum_{a \bmod m \atop \psi(a) = e\left(\frac{\ell}{k}\right)} \sum_{p \leq y \atop p\equiv a \bmod m} \frac{1}{p}.
\end{equation}


In what follows we shall need estimates of Mertens type for sums of reciprocals of primes from specific arithmetic progressions $a$ modulo $m$ that are uniform in a range of the modulus $m$. Results of this type were established by Languasco and Zaccagnini \cite{LaZ}.  
\begin{lem} [Theorem 2 and Corollary 3 of \cite{LaZ}] \label{LOGLZ}
Let $x \geq 3$.  Then, uniformly in $m \leq \log x$ and reduced residue classes $a$ modulo $m$, we have
$$
-\sum_{p \leq x \atop p \equiv a \bmod m} \log\left(1-\frac{1}{p}\right) = \frac{1}{\phi(m)} \log_2 x -  C_m(a) + O\left(\frac{(\log\log x)^{16/5}}{(\log x)^{3/5}}\right),
$$
where $C_m(a)$ is defined in \eqref{MERTENSCONSTANT} below.
\end{lem}


We shall refer to $C_m(a)$ as the \emph{Mertens constant} of the residue class $a$ modulo $m$. Given $m \geq 2$ and $(a, m)=1$, this quantity is defined by
\begin{equation}\label{MERTENSCONSTANT}
C_m(a):= \frac{1}{\phi(m)}\sum_{\chi\neq \chi_0 \bmod m} \overline{\chi}(a) \cdot \log \frac{K(1, \chi)}{L(1, \chi)}-\frac{1}{\phi(m)}\left(\gamma + \log(\phi(m)/m)\right),
\end{equation}
where, for each non-principal character $\chi$ modulo $q$,
$$K(s, \chi):= \sum_{n=1}^{\infty} \frac{k_{\chi}(n)}{n^s} $$ 
is an absolutely convergent Dirichlet series for $\re(s)>0$, and $k_{\chi}(n)$ is a completely multiplicative function defined as
\begin{equation}\label{KChi}
k_{\chi}(p):= p\left(1-\left(1-\frac{\chi(p)}{p}\right)\left(1-\frac 1p\right)^{-\chi(p)}\right).
\end{equation}
Moreover, it follows from the definition of $k_{\chi}(p)$ and Taylor expansion that
\begin{equation}\label{BoundK}
|k_{\chi}(p)|\ll \frac{1}{p}.
\end{equation}

In order to study the asymptotic behaviour of the sum in \eqref{ALLAS}, it will be crucial to have an upper bound for the average of $|C_m(a)|$. 
\begin{lem}\label{AverageMertens}
Fix $\varepsilon>0$, and let $m\geq 3$. Then, we have
$$
\sum_{\substack{a \bmod m\\ (a,m)=1}} \left|C_m(a)\right| \leq  \begin{cases} O(\log_2m), & \text{ if } m \text{ is a non-exceptional modulus},\\
\varepsilon \log m + O(\log_2m), &  \text{ if } m \text{ is exceptional}.\end{cases}
$$
\end{lem}

\begin{proof}
First, since $\phi(m)\gg m/\log_2m$ then 
$$
C_m(a)=\frac{1}{\phi(m)}\sum_{\chi\neq \chi_0 \bmod m} \overline{\chi}(a) \cdot \log \frac{K(1, \chi)}{L(1, \chi)}+O\left(\frac{\log_3 m}{\phi(m)}\right).
$$
Let $\chi$ be a non-principal character modulo $m$. By \eqref{BoundK} we have
\begin{align*}
\log K(1, \chi)&= -\sum_{p\leq x} \log\left(1-\frac{k_{\chi}(p)}{p}\right) +O\left(\sum_{p>x}\frac{|k_{\chi}(p)|}{p}\right)\\
&= -\sum_{p\leq x} \log\left(1-\frac{k_{\chi}(p)}{p}\right) +O\left(\frac 1x\right).
\end{align*}
Furthermore, it follows from \eqref{KChi} that
$$-\log\left(1-\frac{k_{\chi}(p)}{p}\right)+\log \left(1-\frac{\chi(p)}{p}\right)= \chi(p)\log \left(1-\frac{1}{p}\right).$$
If $\chi$ is a non-exceptional character, then $L(\sigma+it, \chi)$ does not vanish when
$$ \sigma\geq 1-\frac{c}{\log(m(|t|+2))},$$
for some positive constant $c$. 
Therefore, taking $T=m^2$, $\sigma_0= 1- c/(4\log m)$ and $X=\exp((\log m)^{3})$ in Lemma \ref{GoodLFunctions} we obtain
\begin{equation}\label{LongApproxL1}
\log L(1,\chi)= -\sum_{p\leq X}\log\left(1-
\frac{\chi(p)}{p}\right) +O\left(\frac{1}{m}\right).
\end{equation}
 We first consider the case when $m$ is a non-exceptional modulus. Using the above estimates together with the orthogonality of characters we conclude that
\begin{equation}\label{OrthogonalityCM}
\begin{aligned}
C_m(a)&= \frac{1}{\phi(m)}\sum_{\chi\neq \chi_0 \bmod m} \overline{\chi}(a)\sum_{p\leq X}\chi(p)\log \left(1-\frac{1}{p}\right) +O\left(\frac{\log_3 m}{\phi(m)}\right)\\
&= \sum_{\substack{p\leq X\\ p\equiv a \bmod m}}\log \left(1-\frac{1}{p}\right)-\frac{1}{\phi(m)} \sum_{\substack{p\leq X \\ p\nmid m}}\log \left(1-\frac{1}{p}\right)+O\left(\frac{\log_3 m}{\phi(m)}\right).\\
\end{aligned}
\end{equation}
Thus, we deduce in this case that
\begin{align*}
\sum_{\substack{a \bmod m\\ (a,m)=1}} \left|C_m(a)\right|
&\leq 
- \sum_{\substack{a \bmod m\\ (a,m)=1}} \sum_{\substack{p\leq X\\ p\equiv a \bmod m}}\log \left(1-\frac{1}{p}\right)
- \sum_{p\leq X}\log \left(1-\frac{1}{p}\right) +O(\log_3m)\\
& \ll \log_2 m.
\end{align*}
Now, suppose that $m$ is an exceptional modulus, and let $\chi_m$ be the exceptional character modulo $m$. The approximation \eqref{LongApproxL1} is valid for all non-principal characters $\chi\neq \chi_m$ modulo $m$. Furthermore, for $\chi=\chi_m$ we have Siegel's bound (see Theorem 11.4 in \cite{MVbook})
$$ \log L(1, \chi_m)\geq -\varepsilon \log m +O_{\varepsilon}(1),$$
and hence, instead of \eqref{LongApproxL1} we use that
$$ \left|\log L(1, \chi_m) + \sum_{p\leq X}\log\left(1-
\frac{\chi_m(p)}{p}\right)\right|\leq \varepsilon\log m+ O(\log_2 m).$$
Thus, similarly to \eqref{OrthogonalityCM} we obtain in this case that
$$
|C_m(a)| \leq -\sum_{\substack{p\leq X\\ p\equiv a \bmod m}}\log \left(1-\frac{1}{p}\right)+ \frac{\varepsilon\log m}{\phi(m)}+O\left(\frac{\log_2 m}{\phi(m)}\right).$$
Summing over all reduced residue classes $a$ modulo $m$ gives the desired bound.
\end{proof}
Proposition \ref{PreciseSumMax} now follows readily.
\begin{proof}[Proof of Proposition \ref{PreciseSumMax}]
First, note that for each fixed $\ell$ modulo $k$, there are exactly $\phi(m)/k$ residue classes $a$ modulo $m$ such that $(a,m)=1$ and $\psi(a) = e\left(\frac{\ell}{k}\right)$. This follows from the simple fact that the number of such residue classes equals the size of the kernel of $\psi$, and by basic group theory this is $|\left(\mb{Z}/m\mb{Z}\right)^{\ast}|/|\text{Im}(\psi)| = \phi(m)/k$. Thus, we deduce from \eqref{ALLAS} and Lemma \ref{LOGLZ} that
\begin{align*}
\sum_{p \leq y \atop \psi(p) = e\left(\frac{\ell}{k}\right)} \frac{1}{p} &= \sum_{a\bmod m \atop \psi(a) = e\left(\frac{\ell}{k}\right)} \left(\frac{\log_2 y}{\phi(m)} - C_m(a) + \sum_{p \leq y \atop p \equiv a(m)} \left(\log\left(1-\frac{1}{p}\right)+ \frac{1}{p}\right) + O\left(\frac{1}{(\log y)^{4/7}}\right)\right)\\
&= \frac{\log_2 y}{k} - \sum_{a \bmod m \atop \psi(a) = e\left(\frac{\ell}{k}\right)} C_m(a) + \sum_{p \leq y \atop \psi(p) =e\left(\frac{\ell}{k}\right)} \left(\log\left(1-\frac{1}{p}\right) + \frac{1}{p}\right) + O\left(\frac{\phi(m)}{k(\log y)^{4/7}}\right).
\end{align*}
Summing over $\ell$ modulo $k$, and using Lemma \ref{MAX},  we get
\begin{align*}
&\sum_{\ell\bmod k} \max_{z \in \mu_g \cup \{0\}} \textup{Re}\left(z \cdot e\left(-\frac{\ell}{k}\right)\right)\sum_{p \leq y \atop \psi(p)= e\left(\frac{\ell}{k}\right)} \frac{1}{p} \\
&= (1-\delta_g)\frac{\pi/gk^{\ast}}{\tan(\pi/gk^{\ast})}\log_2 y + \theta \sum_{\substack{a \bmod m\\ (a,m)=1}} |C_m(a)|+ O\left(1\right),
\end{align*}
for some complex number $|\theta|\leq 1$. Appealing to Lemma \ref{AverageMertens} completes the proof. 
\end{proof}

Let $\psi$ be any odd character modulo $m$, with even order $k$. In choosing characters $\chi$ of order $g$ and conductor $q\leq N$ that maximize the distance $\mb{D}(\chi,\psi;y)$  with $y\leq (\log N)/10$, we will need to be able to choose the values of $\chi$ at the ``small'' primes $p \leq y$. Using Eisenstein's reciprocity law and the Chinese Remainder Theorem, Goldmakher \cite{GOLD} proved the existence of such characters.

\begin{lem}[Proposition 9.3 of \cite{GOLD}] \label{FIX}
Let $g \geq 3$ be fixed, and $y$ be large. Let $\{z_p\}_p$ be a sequence of complex numbers such that $z_p\in \mu_g \cup \{0\}$ for each prime $p$. There exists a positive integer $q$ such that $$g\prod_{\substack{p\leq y\\ p\nmid g}} p\leq q\leq  2g\prod_{\substack{p\leq y\\ p\nmid g}}p,$$ and a primitive Dirichlet character $\chi$ of order $g$ and conductor $q$ such that $\chi(p) = z_p$ for all $p \leq y$ with $p\nmid g$.
\end{lem}
However, in order to prove Theorem \ref{MCHILOW} we need to find ``many'' such characters, since we must avoid those in the exceptional set of Theorem \ref{GL}, which has size at most $N^{\varepsilon}$.  To this end we prove
\begin{lem} \label{VEC}
Let $N$ be large. Let $g \geq 3$ be fixed. Let $2 \leq y \leq (\log N)/10$, and put $\mbf{z} = (z_p)_{p \leq y} \in (\mu_g \cup \{0\})^{\pi(y)}$. There are  
$$\gg \frac{N^{3/4}}{g^{2\pi(y)+2} \log^2 N}$$ primitive Dirichlet characters $\chi$ of order $g$ and conductor $q\leq N$ such that $\chi(p) = z_p$ for each $p \leq y$ such that $p\nmid g$.
\end{lem}
The special case $\mbf{z}=\mbf{1}=(1, 1, \dots, 1)$ was proved by the first author in Lemma 2.3 of \cite{LAM}, but the proof there does not appear to generalize to all $\mbf{z}\in (\mu_g \cup \{0\})^{\pi(y)}$. However, we will show that one can combine the special case $\mbf{z}=\mbf{1}$ with Lemma \ref{FIX} in order to obtain the general case in Lemma \ref{VEC}.    

\begin{proof}[Proof of Lemma \ref{VEC}]
Let $S_{\mbf{z},g}(N)$ be the set of all characters $\chi$ of order $g$ and conductor $q\leq N$ such that $\chi(p) = z_p$ for all $p\leq y$ with $p\nmid g$. By Lemma \ref{FIX}, there exists $\ell$ and a primitive Dirichlet character $\xi$ of order $g$ and conductor $\ell$ such that $\xi(p) = z_p$ for all $p \leq y$ with $p\nmid g$. Moreover, one has
$$\log \ell =\sum_{p\leq y}\log p+O_g(1)= y(1+o(1)),$$
by the prime number theorem, and hence $\ell\leq N^{1/8}$ by our assumption on $y$. 

On the other hand, Lemma 2.3 of \cite{LAM} implies that there are
$$\gg  \frac{N^{3/4}}{g^{2\pi(y)+2} \log^2 N}$$
primitive Dirichlet characters $\psi_{n}$ of order $g$ and conductor $n$, such that $n=q_1q_2$ where $N^{3/8}<q_1<q_2<2N^{3/8}$ are primes with $p_1\equiv p_2\equiv 1\bmod g$, and such that $\psi_{n}(p)=1$ for all primes $p\leq y$. Now, for any such $n$ we have $(\ell, n)=1$ since $\ell \leq N^{1/8}$, and hence $\psi_{n}\xi$ is a primitive character of order $g$ and conductor $n \ell \leq  N$. Finally observe that $\psi_{n}\xi(p) = z_p$ for each $p \leq y$ such that $p\nmid g$. Thus we deduce that $\psi_{n}\xi\in S_{\mbf{z},g}(N)$ for every character $\psi_{n}$, completing the proof.
 \end{proof}
 We finish this section by proving Proposition \ref{UPPER} and Theorem \ref{MCHILOW}.
 \begin{proof}[Proof of Proposition \ref{UPPER}]
 Let $m$ be a non-exceptional modulus, and $\psi$ be an odd primitive character modulo $m$ with order $k$.  For each $0\leq \ell \leq k-1$, suppose that the maximum of $\textup{Re}\left(ze\left(-\frac{\ell}{k}\right)\right)$ for $z \in (\mu_g \cup \{0\})^{\pi(y)}$ is attained when $z=z_{\ell}$. Then, it follows from Lemma \ref{VEC} that there are at least $\sqrt{N}$ primitive characters $\chi$ of order $g$ and conductor $q\leq N$ such that
$$\sum_{p \leq y} \text{Re}\frac{\chi(p)\bar{\psi}(p)}{p}=\sum_{\ell\bmod k} \text{Re}\left(z_{\ell} \cdot e\left(-\frac{\ell}{k}\right)\right)\sum_{p \leq y \atop \psi(p)= e\left(\frac{\ell}{k}\right)} \frac{1}{p}+O_{g}(1).$$
The desired result then follows from \eqref{DIST} and Proposition \ref{PreciseSumMax}. 
\end{proof}

\begin{proof}[Proof of Theorem \ref{MCHILOW}]
Let $N$ be sufficiently large, and let $y=(\log N)/10$. Let $m$ be a prime number that is also a non-exceptional modulus, such that $\sqrt{\log_3 N}\leq m \leq 2\sqrt{\log_3N}$. One can make such a choice since it is known that there is at most one exceptional prime modulus between $x$ and $2x$ for any $x\geq 2$ (see Chapter 14 of \cite{Da} for a reference). Let $\psi$ be a primitive character modulo $m$ of order $k=\phi(m)=m-1$. Note that such a character is necessarily odd. By Proposition \ref{UPPER}, there are at least $\sqrt{N}/2$ primitive characters of order $g$ and conductor $N^{1/3}\leq q\leq N$ such that 
\begin{equation*}
\mb{D}(\chi,\psi;y)^2 = \left(1-(1-\delta_g)\frac{\pi/gk^{\ast}}{\tan(\pi/gk^{\ast})}\right) \log_2y +O(\log_2 m)= \delta_g \log_3 q + O\left(\log_5 q\right),
\end{equation*}
since $gk^{\ast}\geq k$ and $t/\tan(t)=1+O(t^2)$.
Thus, since $\mb{D}(\chi,\psi;\log q)^2=\mb{D}(\chi,\psi;y)^2+O(1)$, then it follows from Theorem \ref{GL} (with $\varepsilon=1/4$) that there are at least $\sqrt{N}/3$ primitive characters of order $g$ and conductor $N^{1/3}\leq q\leq N$ such that
$$ 
M(\chi) \gg \frac{\sqrt{qm}}{\phi(m)} (\log_2 q)^{1-\delta_g} (\log_4 q)^{O(1)}\gg \sqrt{q} (\log_2 q)^{1-\delta_g} (\log_3 q)^{-\frac14}(\log_4 q)^{O(1)} .
$$
\end{proof}
\section{Estimates for $\mc{M}(\chi\bar{\psi}; y, T)$: Proofs of Propositions \ref{MINDIST} and \ref{AD}}

\subsection{Lower bounds for $\mc{M}(\chi\bar{\psi}; y, T)$ for small twists $T$: Proof of Proposition \ref{MINDIST}}
Let $\chi$ be a primitive character modulo $q$ of odd order $g\geq 3$, and $\psi$ be an odd primitive character of conductor $m$ and order $k$. Let $y\geq \exp(m^{7/(4\alpha)})$ be a real number, and put $z=\exp\left((\log y)^{\alpha}\right)$. Since $m\leq(\log z)^{4/7}$, then it follows from 
Proposition \ref{UPPER2} that for all $x\geq z$ we have
\begin{equation}\label{ProUPPER2}
\begin{aligned}
\mb{D}(\chi, \psi; x)^2 &\geq \left(1-(1-\delta_g)\frac{\pi/gk^{\ast}}{\tan(\pi/gk^{\ast})}\right) \log_2x-\beta \varepsilon \log m+ O\left(\log_2 m\right)\\
&\geq \left(\delta_g+ \frac{\pi^2(1-\delta_g)}{4 (gk^{\ast})^2}\right) \log_2x-\beta \varepsilon \log m+ O\left(\log_2 m\right).
\end{aligned}
\end{equation}
since $g k^{\ast}\geq 6$, and $u/\tan(u)\leq 1-u^2/4$ for $0\leq u\leq \pi/6$. 

Let $t$ be a real number such that $|t|\leq (\log y)^{-\alpha}$. First, if $|t|\leq (\log y)^{-1}$ , then since $p^{-it}=1+O(|t|\log p)$ we obtain
\begin{equation}\label{SMALLT}
\mb{D}(\chi\bar{\psi}, n^{it}; y)^2= \mb{D}(\chi, \psi; y)^2+O\left(|t| \sum_{p \leq y} \frac{\log p}{p}\right)=\mb{D}(\chi, \psi; y)^2+O(1),
\end{equation}
and hence the desired lower bound for $\mb{D}(\chi\bar{\psi}, n^{it}; y)^2$ follows from \eqref{ProUPPER2}. Thus, we can and will assume throughout this subsection that $|t|> (\log y)^{-1}$. 
Furthermore, since $|t| \leq (\log y)^{-\alpha}=1/\log z$, then similarly to \eqref{SMALLT} one has
 $$
\mb{D}(\chi\bar{\psi}, n^{it}; y)^2
=\mb{D}(\chi, \psi; z)^2+
\sum_{z < p \leq y} \frac{1-\text{Re}(\chi(p)\bar{\psi}(p)p^{-it})}{p} +O(1).
$$
Therefore, in view of  \eqref{ProUPPER2}, it is enough to prove the following result in order to deduce  Proposition \ref{MINDIST}.
\begin{pro}\label{MEDIUMPRIMES} Let $\chi$, $\psi$, $y$, $z$ and $t$ be as above. Then we have
 $$\sum_{z < p \leq y} \frac{1-\textup{Re}(\chi(p)\bar{\psi}(p)p^{-it})}{p}\geq \delta_g \log\left(\frac{\log y}{\log z}\right)+O(1).$$
\end{pro}
To establish this result, we will follow the arguments in Section 8 of \cite{GOLD}. We shall need the following lemmas.
\begin{lem}[Lemma 8.3 of \cite{GOLD}]\label{MAXGold}
Let $g\geq 3$ be odd, $k\geq 2$ be even, and $\theta\in \mathbb{R}$. Put $k^{\ast}=k/(g, k)$. Then we have
$$
\frac{1}{k}\sum_{\ell \bmod k} \max_{z\in \mu_g\cup \{0\}}\textup{Re}\left(z \cdot e\left(\theta - \frac{\ell}{k}\right)\right) = \frac{\sin(\pi/g)}{ k^{\ast}\tan(\pi/gk^{\ast})} F_{gk^{\ast}}\left(-gk^{\ast} \theta\right),
$$
where $F_n(u) := \cos(2\pi\{u\}/n) + \tan(\pi/n)\sin(2\pi \{u\}/n)$, and $\{u\}$ is the fractional part of $u$.
\end{lem}

\begin{lem}\label{IntegralFN}
Let $T>1$ and $n\geq 3$ be a positive integer. Then
\begin{equation}
\int_1^T \frac{F_{n}(u)}{u} du= \frac{n}{\pi} \tan\left(\frac{\pi}{n}\right) \log T + O(1), \label{UPINT}
\end{equation}
and 
\begin{equation}
\int_{1/T}^1 \frac{F_{n}(u)}{u} du= \log T+ O(1). \label{LOWINT}
\end{equation}
In particular, for any $0<A<B$ we have
\begin{equation}\label{LOWUPINT}
\int_{A}^B \frac{F_{n}(u)}{u} du\leq  \frac{n}{\pi} \tan\left(\frac{\pi}{n}\right)\log(B/A)+ O(1),
\end{equation}
and the constants in the $O(1)$ error terms are absolute. 
\end{lem}
\begin{proof}
We first prove \eqref{UPINT}. Since $F_{n}$ is bounded and $1$-periodic, we have
\begin{align*}
\int_1^T \frac{F_{n}(u)}{u} du 
&= \sum_{1 \leq j \leq T} \int_0^1 \frac{F_{n}(u)}{u+j} du +O(1) = \sum_{1 \leq j \leq T} \frac{1}{j} \int_0^1 F_{n}(u) du + O(1) \\
&= \frac{n}{\pi} \tan\left(\frac{\pi}{n}\right) \log T + O(1). 
\end{align*}
The second estimate \eqref{LOWINT} follows from observing that for $u\in [0, 1)$ and $n\geq 3$ we have
$$ F_n(u)= 1+ O\left(\frac{u^2}{n^2}+\tan\left(\frac{\pi}{n}\right)\frac{u}{n}\right)=1+O(u).$$
Finally, to prove \eqref{LOWUPINT} we consider the three cases $1\leq A<B$, $A<1<B$, and $A<B\leq 1$. The first case follows from \eqref{UPINT}, and the third follows from \eqref{LOWINT} upon using the inequality $\tan(\pi/n)\ge \pi/n$.  Finally, in the second case we have
$$ \int_A^B \frac{F_{n}(u)}{u} du = \int_A^1 \frac{F_{n}(u)}{u} du +\int_1^B\frac{F_{n}(u)}{u} du=\frac{n}{\pi} \tan\left(\frac{\pi}{n}\right)\log B-\log A+ O(1),$$
which implies the result since $\tan(\pi/n)\ge \pi/n$ and $-\log A>0.$
\end{proof}
\begin{proof}[Proof of Proposition \ref{MEDIUMPRIMES}]
Let $x_0=z$, and $\delta>0$ be a small parameter to be chosen. For each positive integer $r \leq R := \left\lfloor\log(y/z)/\log(1+\delta)\right\rfloor$, set $x_r := (1+\delta)^r z$. We consider the sum
$$  S= \sum_{z < p \leq y}\frac{\text{Re}(\chi(p)\bar{\psi}(p)p^{-it})}{p}=\sum_{0 \leq r \leq R-1} \sum_{x_r < p \leq x_{r+1}} \frac{\text{Re}(\chi(p)\bar{\psi}(p)p^{-it})}{p} + O\left(\delta\right).
$$
Write $\theta_r := -\frac{t \log x_r}{2\pi}$, and note that if $p \in (x_r,x_{r+1}]$ then $$|p^{-it} - e(\theta_r)| \ll |t| \log(1+\delta) \ll \delta|t|,$$ so that
\begin{equation}\label{SPLIT}
S= \sum_{0 \leq r \leq R-1} \sum_{x_r < p \leq x_{r+1}} \frac{\text{Re}(e(\theta_r)\chi(p)\bar{\psi}(p))}{p}
 +O\left(\delta\right).
\end{equation}
For each $0 \leq r \leq R-1$, we define
$$
S_r:=\sum_{x_r < p \leq x_{r+1}} \frac{\text{Re}(e(\theta_r)\chi(p)\bar{\psi}(p))}{p}
\leq  \sum_{\ell \bmod k} \max_{z \in \mu_p \cup \{0\}} \text{Re}\left(ze\left(\theta_r-\frac{\ell}{k}\right)\right)\sum_{a \bmod m \atop \psi(a) = e\left(\frac{\ell}{k}\right)} \sum_{x_r < p \leq x_{r+1} \atop p \equiv a \bmod m} \frac{1}{p}.
$$
Note that $m  \leq (\log x_r)^{4/7}$ for each $0 \leq r \leq R$.  Thus, by the Siegel-Walfisz theorem (see Corollary 11.19 in \cite{MVbook}), there is a positive constant $b$ such that 
\begin{equation*}
\sum_{x_r < p \leq x_{r+1} \atop p \equiv a \bmod m} \log p = \frac{x_{r+1}-x_r}{\phi(m)} +O\left(x_r \exp\left(-b\sqrt{\log x_r}\right)\right),
\end{equation*}
for all $0 \leq r \leq R-1$. Moreover, for $x_r < p \leq x_{r+1} = (1+\delta)x_r$, we have
$$
\frac{1}{p} = \frac{\log p}{x_r\log x_r}\left(1+\frac{p\log p-x_r\log x_r}{x_r\log x_r}\right)^{-1}= \big(1+O(\delta)\big) \frac{\log p}{x_r\log x_r}.
$$
Thus, combining these two statements, we get
\begin{equation*}
\sum_{x_r < p \leq x_{r+1} \atop p \equiv a \bmod m} \frac{1}{p} = \frac{1+O(\delta)}{x_r\log x_r} \sum_{x_r < p \leq x_{r+1} \atop p \equiv a \bmod m} \log p = \big(1+O(\delta)\big)\frac{\delta}{\phi(m) \log x_r} +O\left(\exp\left(-b\sqrt{\log z}\right)\right),
\end{equation*}
and upon summing over $a$ modulo $m$ such that $\psi(a) = e\left(\frac{\ell}{k}\right)$, of which there are $\phi(m)/k$ (as remarked in Section 3), we see that
\begin{align*}
S_r &\leq \big(1+O(\delta)\big)\frac{\delta}{k\log x_r}\sum_{\ell \bmod k} \max_{z \in \mu_p \cup \{0\}} \text{Re}\left(ze\left(\theta_r-\frac{\ell}{k}\right)\right) + O\left(\phi(m)\exp\left(-b\sqrt{\log z}\right)\right)\\
& \leq \big(1+O(\delta)\big)\frac{\delta \sin(\pi/g)}{ k^{\ast}\tan(\pi/gk^{\ast})}\frac{F_{gk^{\ast}}(-gk^{\ast} \theta_r)}{\log x_r}+ O\left(\phi(m)\exp\left(-b\sqrt{\log z}\right)\right)
\end{align*} 
by Lemma \ref{MAXGold}. Summing over $0 \leq r \leq R-1$ this yields
\begin{equation}\label{SPLIT2}
\sum_{0 \leq r \leq R-1} S_r \leq\big(1+O(\delta)\big)\frac{\delta \sin(\pi/g)}{ k^{\ast}\tan(\pi/gk^{\ast})} \sum_{0 \leq r \leq R-1} \frac{F_{gk^{\ast}}\left(\frac{t gk^{\ast}}{2\pi} \log x_r\right)}{\log x_r} +  O\left(\exp\left(-(\log y)^{\frac{\alpha}{4}}\right)\right),
\end{equation} 
since $\phi(m) R\ll (\log y)^3,$ and $z=\exp\left((\log y)^{\alpha}\right).$ 

Recall that for $n\ge 3$, $F_n(u)= \cos(2\pi\{u\}/n) + \tan(\pi/n)\sin(2\pi \{u\}/n)$ is bounded, periodic with period $1$, and continuous on $\mathbb{R}$ (since $\lim_{u\to 1^{-}} F_n(u)=F_n(0)$). Moreover, $F_n$ is continuously differentiable on the interval $(0,1)$, and $F_n'(u) =O(1/n)$ uniformly in $u\in \mathbb{R}\setminus\mathbb{Z}$. It follows from  these facts, together with the mean value theorem, that $|F_n(a)-F_n(b)|=O(|a-b|/n)$ for all $a, b\in \mathbb{R}$ such that $|a-b|<1$, where the constant in the $O$ is absolute. This shows that for all $u\in [\log x_r, \log x_{r+1})$ we have
$$ F_{gk^{\ast}}\left(\frac{t gk^{\ast}}{2\pi} u\right)= F_{gk^{\ast}}\left(\frac{t gk^{\ast}}{2\pi} \log x_r\right) +O\big(\delta|t|\big).$$
Furthermore, we note that
$$\int_{\log x_r}^{\log x_{r+1}}\frac{du}{u}=\big(1+O(\delta)\big)\frac{\delta}{\log x_r}.$$
Combining these two facts, we obtain
\begin{align*}
&\frac{\delta}{\log x_r}F_{gk^{\ast}}\left(\frac{t gk^{\ast}}{2\pi} \log x_r\right)
=\big(1+O(\delta)\big)\int_{\log x_r}^{\log x_{r+1}}F_{gk^{\ast}}\left(\frac{t gk^{\ast}}{2\pi} \log x_r\right)\frac{du}{u}\\
&= \big(1+O(\delta)\big)\int_{\log x_r}^{\log x_{r+1}}F_{gk^{\ast}}\left(\frac{t gk^{\ast}}{2\pi} u\right)\frac{du}{u} +O\left(\delta |t| \int_{\log x_r}^{\log x_{r+1}}\frac{du}{u} \right).
\end{align*}
Summing over $0\leq r\leq R-1$, we get
$$
 \delta\sum_{0\leq r\leq R-1}\frac{F_{gk^{\ast}}\left(\frac{t gk^{\ast}}{2\pi} \log x_r\right)}{\log x_r}
 = \big(1+O(\delta)\big)\int_{\log z}^{\log y}F_{gk^{\ast}}\left(\frac{t gk^{\ast}}{2\pi} u\right)\frac{du}{u} +O\left(\delta\right),
$$
since $\int_{\log x_R}^{\log y} dt/t\ll \delta$. 

We now estimate the integral in the main term above. One can easily check that for $n\geq 3$, $F_n$ is an even function. Making the change of variable $v := \frac{gk^{\ast}|t|}{2\pi} u$, and setting $A := \frac{gk^{\ast}|t|}{2\pi} \log z$ and $B := \frac{gk^{\ast}|t|}{2\pi} \log y$, we get
\begin{equation*}
\int_{\log z}^{\log y} F_{gk^{\ast}}\left(\frac{t gk^{\ast}}{2\pi} u\right)\frac{du}{u} = \int_A^B F_{gk^{\ast}}(v) \frac{dv}{v}\leq \frac{gk^{\ast}}{\pi} \tan\left(\frac{\pi}{gk^{\ast}}\right)\log\left(\frac{\log y}{\log z}\right)+ O(1),
\end{equation*}
by Lemma \ref{IntegralFN}. Combining the above estimates with \eqref{SPLIT} and \eqref{SPLIT2} we obtain
$$ S\leq \big(1+O(\delta)\big) \frac{g}{\pi} \sin\left(\frac{\pi}{g}\right)\log\left(\frac{\log y}{\log z}\right)+ O(1)\leq (1-\delta_g)\log\left(\frac{\log y}{\log z}\right)+O(\delta\log_2 y).$$
Choosing $\delta=(\log_2y)^{-1}$ completes the proof of Proposition \ref{MEDIUMPRIMES}. Proposition \ref{MINDIST} follows upon combining this result with \eqref{ProUPPER2}.
\end{proof}

\subsection{Estimating $\mc{M}(\chi\bar{\psi}; y, T)$ for large twists $T$} In this subsection, we prove the following result which implies Proposition \ref{AD}. 
\begin{pro} \label{AD2}
Assume GRH.  Let $g\geq 3$ be a fixed odd integer. Let $N$ be large and $y\leq (\log N)/10$. Let $\psi$ be an odd primitive character of conductor $m$ such that  $\exp\left(2\sqrt{\log_3 y}\right) \leq m\leq \exp\left(\sqrt{\log y}\right)$. Then, there exist at least $\sqrt{N}$  primitive characters $\chi$ of order $g$ and conductor $q\leq N$ such that 
\begin{equation*}
\mb{D}^2(\chi\bar{\psi}, n^i;y) = \delta_g \log_2 y + O\left(\log_2 m\right).
\end{equation*}
\end{pro}
\begin{proof}
We follow the proof of Proposition \ref{MEDIUMPRIMES} in such a way that we achieve equality in all steps. Since the arguments here are similar to those in that proof, we omit some of the details. \\
Let $z := \exp\left((\log m)^2\right)$ and $y \geq z$. Let $\delta>0$ be a small parameter to be chosen and put $R := \left\lfloor \log(y/z)/\log(1+\delta)\right \rfloor$ as before. Set $x_0=z$ and $x_r := (1+\delta)^r x_0$. Then, as $\sum_{p \leq z} \frac{1}{p} \ll \log_2 m$, it suffices to find at least $\sqrt{N}$ primitive characters $\chi$ of order $g$ and conductor $q\leq N$ such that 
\begin{equation} \label{REST}
\sum_{z < p \leq y} \frac{\text{Re}(\chi(p)\bar{\psi}(p)p^{-i})}{p} = (1-\delta_g) \log(\log y/\log z) + O(1).
\end{equation}
Let $\theta_r := -\frac{\log x_r}{2\pi}$, for each $0 \leq r \leq R-1$. As in the proof of Proposition \ref{MEDIUMPRIMES}, when $x_r < p \leq x_{r+1}$ we approximate $p^i$ by $x_r^i$, for each $0 \leq r \leq R-1$. Let $k$ be the order of $\psi$, and for each $r$ let $\{z_{r,\ell}\}_{\ell} \in (\mu_g \cup \{0\})^k$ be chosen so as to maximize the sum
\begin{equation*}
\sum_{\ell \bmod k} \text{Re}\left(z_{r,\ell} \cdot e\left(\theta_r-\frac{\ell}{k}\right)\right)\sum_{a \bmod m \atop \psi(a) =e(\ell/k)} \sum_{x_r < p \leq x_{r+1} \atop p \equiv a \bmod m} \frac{1}{p}.
\end{equation*}
By Lemma \ref{VEC} there are at least $\sqrt{N}$ primitive characters $\chi$ of order $g$ and conductor $q\leq N$ such that $\chi(p) = z_{r,\ell}$ whenever  $x_r < p \leq x_{r+1}$, $\psi(p) = e\left(\ell/k\right)$ and $p\nmid g$. For such characters, it follows that
\begin{align*}
&\sum_{z < p \leq y} \frac{\text{Re}(\chi(p)\bar{\psi}(p)p^{-i})}{p} = \sum_{0 \leq r \leq R-1} \sum_{x_r < p \leq x_{r+1}} \frac{\text{Re}(\chi(p)\bar{\psi(p)} x_r^{-i})}{p} + O\left(\delta\right) \\
&= \sum_{0 \leq r \leq R-1}\sum_{\ell \bmod k} \text{Re}\left(z_{r,\ell} \cdot e\left(\theta_r-\frac{\ell}{k}\right)\right)\sum_{a \bmod m \atop \psi(a) =e(\ell/k)} \sum_{x_r < p \leq x_{r+1} \atop p \equiv a \bmod m} \frac{1}{p} + O_{g}\left(1\right).
\end{align*}
Let 
$$S_r:= \sum_{\ell \bmod k} \text{Re}\left(z_{r,\ell} \cdot\left(\theta_r-\frac{\ell}{k}\right)\right)\sum_{a \bmod m \atop \psi(a) =e(\ell/k)} \sum_{x_r < p \leq x_{r+1} \atop p \equiv a \bmod m} \frac{1}{p}.$$
To estimate the inner sum, we use the following asymptotic formula, which is valid under the assumption of GRH:
\begin{equation*}
\sum_{x_r < p \leq x_{r+1} \atop p \equiv a \bmod m} \log p = \frac{x_{r+1}-x_r}{\phi(m)} + O\left(x_{r}^{1/2}\log^2 x_r\right).
\end{equation*}
This yields
\begin{equation*}
\sum_{x_r < p \leq x_{r+1} \atop p \equiv a \bmod m} \frac{1}{p} = \frac{\delta}{\phi(m) \log x_r} \left(1+O\left(\delta\right)\right) + O\left(x_{r}^{-2/5}\right).
\end{equation*}
Using this estimate and proceeding exactly as in the proof of 
Proposition \ref{MEDIUMPRIMES}, we obtain that 
$$
\sum_{0 \leq r \leq R-1} S_r =\left(1+O\left(\delta\right)\right)\frac{\sin(\pi/g)}{k^{\ast}\tan(\pi/gk^{\ast})}\int_{\log z}^{\log y} \frac{F_{gk^{\ast}}\left(\frac{gk^{\ast}}{2\pi} u\right)}{u} du + O\left(\mathcal{E} \right),
$$
where 
$$
\mathcal{E}\ll \delta + z^{-2/5} \sum_{0 \leq r \leq R-1} (1+\delta)^{-2r/5} \ll \delta  + z^{-2/5}\delta^{-1}.
$$
Here, note that if we transform the integral as we did in the proof of Proposition \ref{MEDIUMPRIMES}, i.e., with $v := \frac{gk^{\ast} u}{2\pi}$ then the bounds of integration, $A := \frac{gk^{\ast} \log z}{2\pi}$ and $B:= \frac{gk^{\ast} \log y}{2\pi}$ are both larger than $1$. Thus, applying Lemma \ref{IntegralFN}, we get
\begin{align*}
\int_{\log z}^{\log y} \frac{F_{gk^{\ast}}\left(\frac{gk^{\ast}}{2\pi} u\right)}{u} du &= \int_A^B \frac{F_{gk^{\ast}}(v)}{v} dv = \int_1^B \frac{F_{gk^{\ast}}(v)}{v} dv - \int_1^A \frac{F_{gk^{\ast}}(v)}{v} dv \\
&= \frac{gk^{\ast}}{\pi}\tan(\pi/gk^{\ast}) \log(B/A) + O(1).
\end{align*}
Inserting this into our estimate for $\sum_r S_r$, we get
\begin{equation*}
\sum_{0 \leq r \leq R-1} S_r = (1-\delta_g)\log(\log y/\log z) + O\left(1+\delta \log_2 y + z^{-\frac{2}{5}} \delta^{-1} \right).
\end{equation*}
Choosing $\delta = (\log_2 y)^{-1}$ as before, and noting that  $z\geq (\log_2 y)^4$ yields \eqref{REST} for $y$ sufficiently large. This completes the proof of Proposition \ref{AD2}. Proposition \ref{AD} follows as well. 
\end{proof}


\section{Logarithmic mean values of completely multiplicative functions: proof of Theorem \ref{LogarithmicMean}}
The key ingredient to the proof of Theorem \ref{LogarithmicMean} is the following generalization of Theorem 2 of  \cite{MV}. 
\begin{thm}\label{MontgomeryVaughan}
Let $f\in \mathcal{F}$ and $x\geq 2$. Then, for any $0<T\leq 1$ we have
$$\sum_{n\leq x} \frac{f(n)}{n}\ll \frac{1}{\log x} \int_{1/\log x}^1 \frac{H_{T}(\alpha)}{\alpha} d\alpha,
$$
where 
$$H_{T}(\alpha)=\left(\sum_{k=-\infty}^{\infty} \max_{s\in \mathcal{A}_{k, T}(\alpha)} \left|\frac{F(1+s)}{s}\right|^2\right)^{1/2}.$$
and 
$$\mathcal{A}_{k,T}(\alpha)=\{s=\sigma+it: \alpha\leq \sigma\leq 1, |t-kT|\leq T/2\}.
$$
\end{thm}
Montgomery and Vaughan \cite{MV} established this result for $T=1$, and a straightforward generalization of their proof allows one to obtain Theorem \ref{MontgomeryVaughan} for any $0<T\leq 1$. For the sake of completeness we will include a full sketch of the necessary modifications to obtain this result. The only different treatment occurs when bounding the integrals on the left hand side of \eqref{TheModificationMV} below.  

\begin{lem}\label{MVModification}
Let $0<\alpha, T\leq 1$. Then we have
\begin{equation}\label{TheModificationMV}
\int_{-\infty}^{\infty}\left|\frac{F'(1+\alpha+it)}{\alpha+it}\right|^2 dt+ \int_{-\infty}^{\infty}\left|\frac{F(1+\alpha+it)}{(\alpha+it)^2}\right|^2 dt \ll \frac{H_{T}(\alpha)^2}{\alpha}.
\end{equation}
\end{lem}
\begin{proof}
First, we have
\begin{align*}
\int_{-\infty}^{\infty}\left|\frac{F'(1+\alpha+it)}{\alpha+it}\right|^2 dt
&=\sum_{k=-\infty}^{\infty}\int_{kT-T/2}^{kT+T/2}\left|\frac{F'(1+\alpha+it)}{\alpha+it}\right|^2dt\\
&\leq \sum_{k=-\infty}^{\infty}\max_{|t-kT|\leq T/2}\left|\frac{F(1+\alpha+it)}{\alpha+it}\right|^2 \int_{kT-T/2}^{kT+T/2}\left|\frac{F'(1+\alpha+it)}{F(1+\alpha+it}\right|^2dt.
\end{align*}
To bound the integral on the right hand side of this inequality, we appeal to a result of Montgomery (see Lemma 6.1 of \cite{Te}) which states that if $\sum_{n\geq 1}a_n n^{-s}$ and $\sum_{n\geq 1}b_n n^{-s}$ are two Dirichlet series which are absolutely convergent for $\re(s)>1$ and satisfy $|a_n|\leq b_n$ for all $n\geq 1$, then we have 
\begin{equation}\label{MontgomeryInequality}
\int_{-u}^{u}\left|\sum_{n=1}^{\infty}\frac{a_n}{n^{\sigma+it}}\right|^2dt \leq 3 \int_{-u}^{u}\left|\sum_{n=1}^{\infty}\frac{b_n}{n^{\sigma+it}}\right|^2 dt,
\end{equation}
for any real numbers $u\geq 0$ and $\sigma>1$. This implies that
\begin{align*}
\int_{kT-T/2}^{kT+T/2}\left|\frac{F'(1+\alpha+it)}{F(1+\alpha+it)}\right|^2dt 
&=\int_{-T/2}^{T/2} \left| \sum_{n=1}^{\infty} \frac{\Lambda(n) f(n)}{n^{1+\alpha+ikT+it}}\right|^2dt 
\ll \int_{-T/2}^{T/2} \left|\frac{\zeta'(1+\alpha+it)}{\zeta(1+\alpha+it)}\right|^2 dt\\
& 
\ll \int_{-T/2}^{T/2} \frac{1}{|\alpha+it|^2}dt\leq \int_{-\infty}^{\infty} \frac{1}{\alpha^2+t^2}dt \ll \frac{1}{\alpha}.
\end{align*}
Hence, we deduce that 
$$ \int_{-\infty}^{\infty}\left|\frac{F'(1+\alpha+it)}{\alpha+it}\right|^2 dt \ll \frac{H_{T}(\alpha)^2}{\alpha}.$$
To complete the proof, note that
$$\int_{-\infty}^{\infty}\left|\frac{F(1+\alpha+it)}{(\alpha+it)^2}\right|^2 dt \leq \sum_{k=-\infty}^{\infty}\max_{|t-kT|\leq T/2}\left|\frac{F(1+\alpha+it)}{\alpha+it}\right|^2 \int_{kT-T/2}^{kT+T/2}\frac{1}{|\alpha+it|^2}dt \ll \frac{H_{T}(\alpha)^2}{\alpha}.$$
\end{proof}
\begin{proof}[Proof of Theorem \ref{MontgomeryVaughan}]
Let
$$S(x)=\sum_{n\leq x} \frac{f(n)}{n}.$$
From the Euler product, $|F(2)|>0$, so $H_{T}(\alpha)\gg 1$. Thus, it is enough to prove the statement for $x\geq x_0$, where $x_0$ is a suitably large constant. Moreover, observe that $\int_{1/\log x}^1H_{T}(\alpha)\alpha^{-1} d\alpha$ is strictly increasing as a function of $x$, and $|S(x)\log x|$ is strictly increasing for $x\in [n, n+1)$, for all $n\geq 1$. Hence it is enough to prove the result for $x\in \mathcal{B}$ where
$$\mathcal{B}=\{ x\geq x_0: |S(y)\log y| < |S(x)\log x| \text{ for all } y<x\}.$$
Montgomery and Vaughan proved that for $x\in \mathcal{B}$ we have (see equations (7) and (8) of \cite{MV})
$$
|S(x)|\log x \ll \int_e^x \frac{|S(u)|}{u}du+ \frac{1}{\log x} \left|\sum_{n\leq x} \frac{f(n)}{n}(\log n)\log\left(\frac xn\right)\right| +  \frac{1}{\log x} \left|\sum_{n\leq x} \frac{f(n)}{n}\log^2\left(\frac xn\right)\right|.
$$
Integrating the first integral by parts, we get
\begin{equation}\label{IntegrationPartS} \int_e^x \frac{|S(u)|}{u}du \ll \frac{J(x)}{\log x} +\int_{e}^x \frac{J(u)}{u(\log u)^2} du,
\end{equation}
where
$$J(u):= \int_{e}^u \frac{|S(t)|\log t}{t} dt \ll (\log u)^{1/2} \left(\int_e^u \frac{|S(t)|^2(\log t)^2}{t} dt\right)^{1/2},
$$
by the Cauchy-Schwarz inequality. Using Parseval's Theorem, Montgomery and Vaughan proved that (see equation (14) of \cite{MV}) 
$$\int_e^u \frac{|S(t)|^2(\log t)^2}{t} dt \ll \int_{-\infty}^{\infty}\left|\frac{F'(1+\beta+it)}{\beta+it}\right|^2 dt+ \int_{-\infty}^{\infty}\left|\frac{F(1+\beta+it)}{(\beta+it)^2}\right|^2 dt,$$
where $\beta=2/\log u$. Appealing to Lemma \ref{MVModification} and making the change of variable $\alpha=1/\log u$ in the integral of the right hand side of \eqref{IntegrationPartS} we deduce that
\begin{equation}\label{OneIntegralTerm}
\int_e^x \frac{|S(u)|}{u}du \ll H_{T}\left(\frac{2}{\log x}\right)+ \int_{1/\log x}^1\frac{H_{T}(2\alpha)}{\alpha} d\alpha.
\end{equation}
Since $H_{T}(\alpha)$ is decreasing as a function of $\alpha$, we have
\begin{equation}\label{OneIntegralTerm2}
H_{T}\left(\frac{2}{\log x}\right)\ll \int_{1/\log x}^{2/\log x} \frac{H_{T}(\alpha)}{\alpha} d\alpha \leq \int_{1/\log x}^1\frac{H_{T}(\alpha)}{\alpha} d\alpha.
\end{equation}
Combining \eqref{OneIntegralTerm} and \eqref{OneIntegralTerm2} we get
$$\int_e^x \frac{|S(u)|}{u}du\ll \int_{1/\log x}^1\frac{H_{T}(\alpha)}{\alpha} d\alpha.$$
Furthermore, Montgomery and Vaughan proved that (see pages 207-208 of \cite{MV})
$$ \sum_{n\leq x} \frac{f(n)}{n}(\log n)\log\left(\frac xn\right) \ll \left(\frac{1}{\beta} \int_{-\infty}^{\infty}\left|\frac{F'(1+\beta+it)}{\beta+it}\right|^2 dt\right)^{1/2}
$$
and 
$$  \sum_{n\leq x} \frac{f(n)}{n}\log^2\left(\frac xn\right)\ll \left(\frac{1}{\beta} \int_{-\infty}^{\infty}\left|\frac{F(1+\beta+it)}{(\beta+it)^2}\right|^2 dt\right)^{1/2},
$$
where $\beta=2/\log x$. Combining these bounds with Lemma \ref{MVModification} and equation \eqref{OneIntegralTerm2} completes the proof.
\end{proof}
In order to derive Theorem \ref{LogarithmicMean} from Theorem \ref{MontgomeryVaughan}, we need to bound $H_{T}(\alpha)$, and hence to bound $|F(1+s)|$ for $\re(s)\geq \alpha$. 
Tenenbaum (see Section III.4 of \cite{Te}) proved that for all $y, T\geq 2$, and $1/\log y\leq \alpha\leq 1$, we have 
\begin{equation}\label{Tenenbaum}
\max_{|t|\leq T}|F(1+\alpha+it)|\ll (\log y) \exp\big(-\mathcal{M}(f; y, T)\big).
\end{equation}
However, this bound does not hold for all $T>0$ and $1/\log y\leq \alpha\leq 1$. Indeed, taking $f$ to be the M\"obius function $\mu$, $\alpha=1/2$,  $y$ large and $T=1/\log y$ shows that $\max_{|t|\leq T}|F(1+\alpha+it)|\geq |\zeta(3/2)|^{-1}$, while
$$ \mathcal{M}(f; y, T)= \min_{|t|\leq 1/\log y} \sum_{p\leq y} \frac{1+ \re(p^{-it})}{p}=
2\sum_{p\leq y} \frac{1}{p}+O(1)=2\log\log y+O(1),$$ 
and hence the right side of \eqref{Tenenbaum} is $\ll 1/(\log y)$. Nevertheless, using Tenenbaum's ideas, we show that \eqref{Tenenbaum} is valid whenever $T\geq \alpha$. 

\begin{lem}\label{MaxFDistance}
Let $y\geq 2$ and $f \in \mathcal{F}$ such that $f(p)=0$ for $p>y$. Let $F(s)$ be its corresponding Dirichlet series.  Then, for all real numbers $0< \alpha\leq 1$ and $T\geq \alpha$ we have 
$$
\max_{|t|\leq T} \big|F(1+\alpha+it)|\ll (\log y)\exp\big(-\mathcal{M}(f; y, T)\big).
$$
\end{lem}
\begin{proof} Note that
\begin{equation}\label{MFYT}
\mathcal{M}(f; y, T)= \log_2 y- \max_{|t|\leq T} \re \sum_{p\leq y} \frac{f(p)}{p^{1+it}}+O(1).
\end{equation}
We first remark that the result is trivial if $\alpha\leq 1/\log y$, since in this case we have 
$$
\log |F(1+\alpha+it)|= \re \sum_{p\leq y} \frac{f(p)}{p^{1+\alpha+it}}+O(1)= \re \sum_{p\leq y} \frac{f(p)}{p^{1+it}}+O(1),
$$
which follows from the fact that $|p^{\alpha}-1|\ll \alpha\log p$.\\
Now, suppose that $\alpha\geq 1/\log y$ and put $A=\exp(1/\alpha)$. Then we have
$$ \log |F(1+\alpha+it)| =\re \sum_{p\leq y} \frac{f(p)}{p^{1+\alpha+it}}+O(1)= \re \sum_{p\leq A} \frac{f(p)}{p^{1+\alpha+it}}+O(1)=\re \sum_{p\leq A} \frac{f(p)}{p^{1+it}}+O(1),$$
since  $\sum_{p>A} p^{-1-\alpha}\ll 1$ by the prime number theorem. Furthermore, for any $|\beta|\leq \alpha/2$ we have 
$$\sum_{p\leq A} \frac{f(p)}{p^{1+i(t+\beta)}}=\sum_{p\leq A} \frac{f(p)}{p^{1+it}}+O(1),$$
and hence
$$ \max_{|t|\leq T} \big|F(1+\alpha+it)|\ll \max_{|t|\leq (T-\alpha/2)} \exp\left(\re\sum_{p\leq A} \frac{f(p)}{p^{1+it}}\right).
$$
Now, let $|t|\leq T-\alpha/2$ be a real number. Then, we have
\begin{align*}
\int_{t-\alpha/2}^{t+\alpha/2}  \re\left(\sum_{p\leq y} \frac{f(p)}{p^{1+iu}}\right) du
&= \re \sum_{p\leq y}\frac{f(p)}{p^{1+it}}\left(\frac{p^{i\alpha/2}-p^{-i\alpha/2}}{i\log p}\right)\\
&=\alpha \left(\re\sum_{p\leq A} \frac{f(p)}{p^{1+it}}\right) + O\left(\alpha + \sum_{p>A} \frac{1}{p\log p}\right).
\end{align*}
Since $\sum_{p>A} (p\log p)^{-1}\ll \alpha$ by the prime number theorem, we deduce that
$$\re\sum_{p\leq A}\frac{f(p)}{p^{1+it}}=\frac{1}{\alpha} \int_{t-\alpha/2}^{t+\alpha/2} \re\left(\sum_{p\leq y} \frac{f(p)}{p^{1+iu}}\right) du +O(1)\leq \max_{|t|\leq T} \re \sum_{p\leq y} \frac{f(p)}{p^{1+it}}+O(1).$$
Appealing to \eqref{MFYT} completes the proof.
\end{proof}

We finish this section by proving a slightly stronger form of Theorem \ref{LogarithmicMean}, which we shall need to prove Theorems \ref{MCHIUP1} and \ref{MCHIUP2}. One can also show that the following result follows from Theorem \ref{LogarithmicMean}, so it is in fact equivalent to it. 
\begin{thm}\label{LogarithmicMean2}
Let $f\in \mathcal{F}$ and $x, y\geq 2$ be real numbers. Then, for any real number $0< T\leq 1$ we have
$$\sum_{\substack{n\leq x\\ n\in \mathcal{S}(y)}}\frac{f(n)}{n}\ll (\log y) \cdot \exp\big(-\mc{M}(f;y,T)\big)+\frac{1}{T},
$$
where the implicit constant is absolute, and $\mc{S}(y)$ is the set of $y$-friable numbers.
\end{thm}

\begin{proof}
First, observe that the result is trivial if $T\leq 1/\log x$, since we have in this case
$$\sum_{\substack{n\leq x\\ n\in \mathcal{S}(y)}}\frac{f(n)}{n}\ll \sum_{n\leq x}\frac{1}{n}\ll \log x\ll \frac{1}{T}.$$
Now assume that $1/\log x<T\leq 1$. Let $g$ be the completely multiplicative function such that $g(p)=f(p)$ for $p\leq y$ and $g(p) = 0$ otherwise, and let $G$ be its corresponding Dirichlet series. Then, it follows from Theorem \ref{MontgomeryVaughan} that
\begin{equation}\label{SmoothFG}
\sum_{\substack{n\leq x\\ n\in \mathcal{S}(y)}} \frac{f(n)}{n}=
\sum_{n\leq x}\frac{g(n)}{n}\ll \frac{1}{\log x}\int_{1/\log x}^1 \frac{H_{T}(\alpha)}{\alpha} d\alpha, 
\end{equation}
where 
$$
H_{T}(\alpha)=\left(\sum_{k=-\infty}^{\infty} \max_{s\in \mathcal{A}_{k, T}(\alpha)} \left|\frac{G(1+s)}{s}\right|^2\right)^{1/2}.
$$
First, observe that if $|t-kT|\leq T/2$ and $k\neq 0$ then $|t|\asymp |k|T$.  Moreover, uniformly for all $t\in \mathbb{R}$, we have 
\begin{equation}\label{TrivialBoundF}
|G(1+\sigma+it)|\leq \zeta(1+\sigma)\ll \frac{1}{\sigma}.
\end{equation}
We will first bound $H_{T}(\alpha)$ when $\alpha>T$.  Using \eqref{TrivialBoundF} we obtain in this case
\begin{equation}\label{BigAlpha}
\alpha^2 \cdot H_{T}(\alpha)^2\ll \sum_{k=-\infty}^{\infty} \max_{|t-kT|\leq T/2}\frac{1}{\alpha^2 +t^2}\ll \sum_{|k|> \alpha/T}\frac{1}{k^2T^2}+ \sum_{|k|\leq \alpha/T} \frac{1}{\alpha^2}\ll \frac{1}{\alpha T}.
\end{equation}
Now, suppose that $0<\alpha\leq T$. To bound $H_{T}(\alpha)$ in this case, we first use \eqref{TrivialBoundF} for $|k|\geq 1$. This gives
$$
H_{T}(\alpha)^2 
\ll \frac{1}{\alpha^2}\sum_{|k|\geq 1} \frac{1}{k^2T^2}+\frac{1}{\alpha^2}\max_{s\in \mathcal{A}_{0, T}(\alpha)} |G(1+s)|^2\ll \frac{1}{(\alpha T)^2}+ \frac{1}{\alpha^2}\max_{s\in \mathcal{A}_{0, T}(\alpha)} |G(1+s)|^2.
$$
Furthermore, by  \eqref{TrivialBoundF} and Lemma \ref{MaxFDistance} we have
\begin{align*}
\max_{s\in \mathcal{A}_{0, T}(\alpha)} |G(1+s)|
&\ll \max_{\substack{|t|\leq T\\ \sigma\geq T}} |G(1+\sigma+it)|+ \max_{\substack{|t|\leq T\\ \alpha \leq \sigma\leq T}} |G(1+\sigma+it)|\\
&\ll \frac{1}{T}+ (\log y)\exp\big(-\mathcal{M}(g; y, T)\big).
\end{align*}
Since $\mathcal{M}(g; y, T)=\mathcal{M}(f; y, T)$ we deduce that for $0<\alpha\leq T$ we have
\begin{equation}\label{SmallAlpha}
H_{T}(\alpha)^2\ll \frac{1}{(\alpha T)^2}+\frac{(\log y)^2}{\alpha^2}\exp\left(-2\mathcal{M}(f; y, T)\right).
\end{equation}
Using \eqref{BigAlpha} when $T<\alpha\leq 1$ and \eqref{SmallAlpha} when $1/\log x\leq \alpha\leq T$ we get
\begin{align*}
\int_{1/\log x}^1 \frac{H_{T}(\alpha)}{\alpha} d\alpha 
&\ll \left(\frac{1}{T}+ \log y \cdot \exp\left(-\mathcal{M}(f; y, T)\right)\right)\int_{1/\log x}^{T} \frac{1}{\alpha^2}d\alpha+  \frac{1}{T^{1/2}}\int_{T}^1\frac{1}{\alpha^{5/2}}d\alpha\\
&\ll \frac{\log x}{T}+ (\log x)(\log y) \exp\left(-\mathcal{M}(f; y, T)\right). 
\end{align*}
Inserting this bound in \eqref{SmoothFG} yields the result.

\end{proof}
\section{Proofs of Theorems \ref{COND}, \ref{MCHIUP1} and \ref{MCHIUP2}}

To prove Theorem \ref{COND}, the general strategy we use is that of \cite{GrSo2} (with the refinements from \cite{GOLD}), and it will be clear where we shall make use of Theorem \ref{LogarithmicMean2}. We will consider the conditional (on GRH) and unconditional results simultaneously, setting $y := \log^{12} q$ if we are assuming GRH, and setting $y := q$ otherwise. We recall here that $y=Q$ in the unconditional case, and $y = Q^{12}$ on GRH, so that in all cases we have $\log y \asymp \log Q$.\\
 When $\chi$ is primitive and $\alpha \in \mb{R}$, we have
\begin{equation*}
\sum_{n \leq q} \frac{\chi(n)}{n}e(n\alpha) = \sum_{n \leq q \atop n \in \mc{S}(y)}\frac{\chi(n)}{n}e(n\alpha) + O(1);
\end{equation*}
on GRH, this follows from \eqref{APPROXFRIABLE}, and unconditionally this statement is trivial.
Inserting this estimate in P\'olya's Fourier expansion \eqref{Polya} gives
\begin{equation*}
M(\chi)\ll \sqrt{q}\left( \max_{\alpha \in [0,1]} \left|\sum_{1\leq |n| \leq q \atop n \in \mc{S}(y)} \frac{\chi(n)}{n}\big(1-e(n\alpha)\big)\right| + 1\right).
\end{equation*}
Therefore, to prove Theorem \ref{COND} it suffices to show that for all $\alpha\in [0, 1]$ we have 
\begin{equation}\label{BoundExponentialSum}
\sum_{1\leq |n| \leq q \atop n \in \mc{S}(y)}\frac{\chi(n)}{n}e\left(n\alpha\right) \ll \Big(1-\chi(-1)\xi(-1)\Big)\frac{\sqrt{m}}{\phi(m)}(\log Q ) e^{-\mc{M}(\chi\bar{\xi};Q, (\log Q)^{-7/11})}+ \left(\log Q\right)^{\frac{9}{11} + o(1)}.
\end{equation}
Let $\alpha\in [0, 1]$ and $R:= (\log Q)^{5}$. By Dirichlet's theorem on Diophantine approximation, there exists a rational approximation $|\alpha - b/r| \leq 1/rR$, with $1\leq r\leq R$ and $(b, r)=1$. 
Let $M := (\log Q)^{4/11}$. We shall distinguish between two cases. If $r\leq M$, we say that $\alpha$ lies on a \emph{major} arc, and if $M<r\leq R$ we say that $\alpha$ lies on a \emph{minor} arc. In the latter case, we shall use Corollary 2.2 of \cite{GOLD}, which is a consequence of the work of Montgomery and Vaughan \cite{MV2}. Indeed, this shows that
\begin{equation*}
\sum_{1\leq |n| \leq q \atop n \in \mc{S}(y)} \frac{\chi(n)}{n}e\left(n\alpha\right) \ll \frac{(\log M)^{5/2}}{\sqrt{M}}\log y +\log R+\log_2 y \ll \left(\log Q\right)^{\frac{9}{11} + o(1)}.
\end{equation*}
We now handle the more difficult case of $\alpha$ lying on a major arc. First, it follows from Lemma 4.1 of \cite{GOLD} (which is a refinement of Lemma 6.2 of \cite{GrSo2}) that for $N := \min\{q,|r\alpha - b|^{-1}\}$, we have
\begin{equation}\label{DIVSUM}
\begin{aligned}
\sum_{1\leq |n| \leq q \atop n \in \mc{S}(y)} \frac{\chi(n)}{n}e(n\alpha) &= \sum_{1\leq |n| \leq N \atop n \in \mc{S}(y)} \frac{\chi(n)}{n}e\left(\frac{nb}{r}\right) + O\left(\frac{(\log R)^{3/2}}{\sqrt{R}}(\log y)^2 +\log R+\log_2 y\right)\\
&= \sum_{1\leq |n| \leq N \atop n \in \mc{S}(y)} \frac{\chi(n)}{n}e\left(\frac{nb}{r}\right) + O\left(\log_2 Q\right).
\end{aligned}
\end{equation}
We first assume that $b\neq 0$. In this case we can use an identity of Granville and Soundararajan (see Proposition 2.3 of \cite{GOLD}) which asserts that
\begin{equation}\label{GrSoIdentity}
\begin{aligned}
&\sum_{1\leq |n| \leq N \atop n \in \mc{S}(y)} \frac{\chi(n)}{n}e\left(\frac{nb}{r}\right)\\
&=\Big(1-\chi(-1)\psi(-1)\Big)\sum_{d\mid r \atop d \in \mc{S}(y)} \frac{\chi(d)}{d}\cdot \frac{1}{\phi(r/d)} \sum_{\psi \bmod r/d}\tau(\psi) \bar{\psi}(b) 
\left(\sum_{n \leq N/d \atop n \in \mc{S}(y)} \frac{\chi(n)\bar{\psi}(n)}{n}\right).
\end{aligned}
\end{equation}
To bound the inner sum above, we appeal to Theorem \ref{LogarithmicMean2} with $T=(\log Q)^{-7/11}$. This implies that
$$
\sum_{n \leq N/d \atop n \in \mc{S}(y)} \frac{\chi(n)\bar{\psi}(n)}{n}\ll (\log y)\cdot \exp\left(-\mc{M}(\chi\bar{\psi};y, (\log Q)^{-7/11})\right)+(\log Q)^{7/11}.
$$
Moreover, in the conditional case  $y = Q^{12}$, and thus we have
$$ \mc{M}(\chi\bar{\psi};y, (\log Q)^{-7/11})\geq \mc{M}(\chi\bar{\psi};Q, (\log Q)^{-7/11}) +O(1).$$
Therefore, we get
\begin{equation}\label{Thm1.3}
\sum_{n \leq N/d \atop n \in \mc{S}(y)} \frac{\chi(n)\bar{\psi}(n)}{n}\ll (\log Q)\cdot \exp\left(-\mc{M}(\chi\bar{\psi};Q, (\log Q)^{-7/11})\right)+(\log Q)^{7/11}.
\end{equation}
We now order the primitive characters $\psi\pmod \ell$ for $\ell\leq M$ (including the trivial character $\psi$ which equals $1$ for all integers)  as $\{\psi_k\}_k$, where
\begin{equation*}
\mc{M}(\chi\bar{\psi_k};Q, (\log Q)^{-7/11}) \leq \mc{M}(\chi\bar{\psi_{k+1}};Q, (\log Q)^{-7/11}), 
\end{equation*}
for all $k \geq 1$. Note that $\psi_1=\xi$, in the notation of Theorem \ref{COND}. Furthermore, by a slight variation of Lemma 3.1 of \cite{BGS} we have  
$$
 \mc{M}\left(\chi\bar{\psi_k};Q, (\log Q)^{-7/11}\right) \geq 
\left(1-\frac{1}{\sqrt{k}}\right)\log_2 Q+O\left(\sqrt{\log_2 Q}\right).
$$
Therefore, if $\psi \pmod \ell$ is induced by $\psi_{k}$, then
\begin{equation}\label{Repulsive}
\begin{aligned}
\mc{M}\left(\chi\bar{\psi};Q, (\log Q)^{-7/11}\right) &\geq \mc{M}\left(\chi\bar{\psi_k};Q, (\log Q)^{-7/11}\right) +O\left(\sum_{p \mid \ell} \frac{1}{p}\right)\\
 &\geq 
\left(1-\frac{1}{\sqrt{k}}+o(1)\right)\log_2 Q,
\end{aligned}
\end{equation}
since $\sum_{p\mid \ell}1/p\ll \log_2 \ell\ll \log_3 Q$. 
Inserting this bound in \eqref{Thm1.3}, we deduce that the contribution of all characters $\psi$ that are induced by some $\psi_k$ with $k\geq 3$ to \eqref{GrSoIdentity} is
$$
\ll (\log Q)^{7/11}\sum_{d\mid r} \frac{1}{d\phi(r/d)} \sum_{\psi \bmod r/d}|\tau(\psi)|\ll  (\log Q)^{7/11}\sum_{d\mid r} \frac{\sqrt{r}}{d^{3/2}}\ll (\log Q)^{9/11},
$$
since $1/\sqrt{3}<7/11$, $|\tau(\psi)|\leq \sqrt{r/d}$, and $r\leq (\log Q)^{4/11}.$ Moreover, 
observe that there is at most one character $\psi \pmod{r/d}$ such that $\psi$ is induced by $\psi_2$. Using \eqref{Repulsive}, we deduce that the contribution of these characters to  \eqref{GrSoIdentity} is
$$ \ll (\log Q)^{1/\sqrt{2}+o(1)}\sum_{d\mid r} \frac{1}{d}\cdot\frac{\sqrt{r/d}}{\phi(r/d)} \ll (\log Q)^{1/\sqrt{2}+o(1)}\log r \ll(\log Q)^{1/\sqrt{2}+o(1)}. $$
Thus, it now remains to estimate the contribution of the characters $\psi\bmod r/d$ that are induced by $\xi$, recalling that $\xi$ has conductor $m$. If $m\nmid r$, there are no such characters $\psi$ and the theorem follows in this case. If $m\mid r$ and $\psi\bmod r/d$ is induced by $\xi$, then we must have $d \mid (r/m)$. Furthermore, by Lemma 4.1 of \cite{GrSo2} we have 
$$\tau(\psi)= \mu\left(\frac{r}{dm}\right)\xi\left(\frac{r}{dm}\right)\tau(\xi).$$
Therefore, the contribution of these characters to \eqref{GrSoIdentity} is
\begin{equation}\label{INDUCED}
\Big(1-\chi(-1)\xi(-1)\Big)\bar{\xi}(b)\tau(\xi)\sum_{d\mid (r/m) \atop d \in \mc{S}(y)} \frac{\chi(d)}{d}\cdot \frac{1}{\phi(r/d)}  \mu\left(\frac{r}{dm}\right)\xi\left(\frac{r}{dm}\right)
\sum_{\substack{n \leq N/d\\ (n, r/d)=1\\ n \in \mc{S}(y)}} \frac{\chi(n)\bar{\xi}(n)}{n}.
\end{equation}
Furthermore, it follows from Lemma 4.4 of \cite{GrSo2} that
\begin{align*}
\sum_{\substack{n \leq N/d\\ (n, r/d)=1\\ n \in \mc{S}(y)}} \frac{\chi(n)\bar{\xi}(n)}{n} &= \sum_{\substack{n \leq N\\ (n, r/d)=1\\ n \in \mc{S}(y)}} \frac{\chi(n)\bar{\xi}(n)}{n} +O(\log d)\\
&= \prod_{p \mid  \frac rd} \left(1-\frac{\chi(p)\bar{\xi}(p)}{p}\right) \sum_{n \leq N \atop n \in \mc{S}(y)} \frac{\chi(n)\bar{\xi}(n)}{n} + O\left((\log_2 Q)^2\right).
\end{align*}
Thus, in view of Theorem \ref{LogarithmicMean2}, we deduce that \eqref{INDUCED} is 
\begin{equation}\label{INDUCED2}
\begin{aligned} 
\ll &\Big(1-\chi(-1)\xi(-1)\Big)\sqrt{m}\left((\log Q)e^{-\mc{M}\left(\chi\bar{\xi};Q, (\log Q)^{-7/11}\right)}+(\log Q)^{7/11}\right)\\
& \times \sum_{\substack{d\mid (r/m)\\ (r/(dm), m)=1}} \frac{1}{d\phi(r/d)}  \mu^2\left(\frac{r}{dm}\right) \prod_{p \mid \frac{r}{dm}} \left(1+\frac{1}{p}\right).
\end{aligned}
\end{equation}
Finally, by a change of variables $a= r/(md)$, we obtain
\begin{align*}
\sum_{\substack{d\mid (r/m)\\ (r/(dm), m)=1}} \frac{1}{d\phi(r/d)}  \mu^2\left(\frac{r}{dm}\right) \prod_{p \mid  \frac{r}{dm}} \left(1+\frac{1}{p}\right)
&= \frac{m}{r\phi(m)}\sum_{\substack{a\mid (r/m)\\ (a, m)=1}} \frac{a}{\phi(a)}  \mu^2\left(a\right) \prod_{p \mid  a} \left(1+\frac{1}{p}\right)\\
&\leq \frac{1}{\phi(m)}\cdot \frac{1}{r/m} \prod_{p \mid (r/m)} \left(1+\frac{p+1}{p-1}\right)\leq \frac{4}{\phi(m)}, 
\end{align*}
since $2p/(p-1)\leq p$ for all primes $p\geq 3$. Combining this bound with  \eqref{INDUCED2}, it follows that the contribution of the characters $\psi$ that are induced by $\xi$ to \eqref{GrSoIdentity} is
$$\ll \Big(1-\chi(-1)\xi(-1)\Big)\frac{\sqrt{m}}{\phi(m)}(\log Q) e^{-\mc{M}\left(\chi\bar{\xi};Q, (\log Q)^{-7/11}\right)}+(\log Q)^{7/11}. $$

It thus remains to consider when $b = 0$, and hence $r = 1$. First, if $\xi$ is identically $1$ (so $m=1$), then a trivial application of Theorem \ref{LogarithmicMean2} shows that in this case
\begin{equation*}
\sum_{1\leq |n| \leq N \atop n \in \mc{S}(y)} \frac{\chi(n)}{n} \ll \Big(1-\chi(-1)\Big)\frac{\sqrt{m}}{\phi(m)} (\log Q)e^{-\mc{M}\left(\chi; Q,(\log Q)^{-7/11}\right)} + (\log Q)^{7/11}.
\end{equation*}
On the other hand, if $\xi$ is not the trivial character, then it follows from \eqref{Repulsive} that
$$
\mc{M}(\chi; Q,(\log Q)^{-7/11}) \geq \left(1-\frac{1}{\sqrt{2}}+o(1)\right)\log_2 Q,
$$
and hence by Theorem \ref{LogarithmicMean2} we get
$$ 
\sum_{1\leq |n| \leq N \atop n \in \mc{S}(y)} \frac{\chi(n)}{n}\ll (\log Q)^{1/\sqrt{2}+o(1)},
$$
which completes the proof of \eqref{BoundExponentialSum}. Theorem \ref{COND} follows as well.

We end this section by deducing Theorems \ref{MCHIUP1} and \ref{MCHIUP2} from Theorem \ref{COND} and Proposition \ref{MINDIST}. We shall prove both results simultaneously, by setting $Q:=\log q$ on GRH and $Q:=q$ unconditionally. 
\begin{proof}[Proof of Theorems \ref{MCHIUP1} and \ref{MCHIUP2}]
Let  $\xi$ be the character of conductor $m \leq (\log Q)^{4/11}$ that minimizes $\mc{M}\left(\chi\bar{\psi}; Q,(\log Q)^{-7/11}\right)$. If $\xi$ is even, then it follows from Theorem \ref{COND} that
$$ M(\chi)\ll \sqrt{q} (\log Q)^{9/11+o(1)},$$
which trivially implies the result in this case since $1-\delta_g>9/11$, for all $g\geq 3$.
Now, suppose that $\xi$ is odd and let $k$ be its order. We also let $\beta=1$ if $m$ is an exceptional modulus, and $\beta=0$ otherwise. Then, combining Theorem \ref{COND} and Proposition \ref{MINDIST} (with $\alpha=7/11$) we obtain
\begin{equation}\label{FINALBOUND}
\begin{aligned}
M(\chi) &\ll \frac{\sqrt{qm}}{\phi(m)} (\log Q)^{1-\delta_g} \exp\left(-\frac{c_1(1-\delta_g)}{(gk^{\ast})^2}\log_2 Q+ \beta \e \log m+  O\left(\log_2 m\right)\right) \\
&\ll \sqrt{q}(\log Q)^{1-\delta_g} \exp\left(-\left(\frac{1}{2}-\beta \e\right) \log m- \frac{c_1(1-\delta_g)}{g^2m^2}\log_2 Q+ c_2 \log_2 m\right),
\end{aligned}
\end{equation}
for some  positive constants $c_1, c_2$, since $\phi(m)\gg m/\log_2 m$.  One can easily check that the expression inside the exponential is maximal when $m \asymp \sqrt{\log_2 Q}$, and its maximum equals 
$$ -\left(\frac{1}{4}-\frac{\beta \e}{2}\right) \log_3 Q+ O\left(\log_4 Q\right).$$
Inserting this estimate in \eqref{FINALBOUND} completes the proof.
\end{proof}

\end{document}